\documentclass[reqno,11pt]{amsart}
\pdfoutput=1
\usepackage[margin=1.24in]{geometry}

\title{Conformally compact metrics and the Lovelock tensors}
\date{\today}
\author{Xinran Yu}
\address{University of Illinois Urbana-Champaign}
\email{xinran4@illinois.edu}

\makeatletter
\renewcommand{\tocsection}[3]{%
  \indentlabel{\@ifnotempty{#2}{\bfseries\ignorespaces#1 #2\quad}}\bfseries#3}
\renewcommand{\tocsubsection}[3]{%
  \indentlabel{\@ifnotempty{#2}{\ignorespaces#1 #2\quad}}#3}
\renewcommand{\tocsubsubsection}[3]{%
  \indentlabel{\@ifnotempty{#2}{\ignorespaces#1 #2\quad}}#3}
\newcommand\@dotsep{4.5}
\def\@tocline#1#2#3#4#5#6#7{\relax
  \ifnum #1>\c@tocdepth % then omit
  \else
    \par \addpenalty\@secpenalty\addvspace{#2}%
    \begingroup \hyphenpenalty\@M
    \@ifempty{#4}{%
      \@tempdima\csname r@tocindent\number#1\endcsname\relax
    }{%
      \@tempdima#4\relax
    }%
    \parindent\z@ \leftskip#3\relax \advance\leftskip\@tempdima\relax
    \rightskip\@pnumwidth plus1em \parfillskip-\@pnumwidth
    #5\leavevmode\hskip-\@tempdima{#6}\nobreak
    \leaders\hbox{$\m@th\mkern \@dotsep mu\hbox{.}\mkern \@dotsep mu$}\hfill
    \nobreak
    \hbox to\@pnumwidth{\@tocpagenum{\ifnum#1=1\bfseries\fi#7}}\par% <-- \bfseries for \section page
    \nobreak
    \endgroup
  \fi}
\AtBeginDocument{%
\expandafter\renewcommand\csname r@tocindent0\endcsname{0pt}
}
\def\l@subsection{\@tocline{2}{0pt}{2.5pc}{5pc}{}}
\def\l@subsubsection{\@tocline{2}{0pt}{4.5pc}{5pc}{}}

\makeatother

\newcommand{\nocontentsline}[3]{}

\usepackage{amsmath}
\usepackage{amsfonts, mathrsfs} % mathbb, mathscr
\usepackage{mathtools}

\usepackage{tensor}
\usepackage{slashed} % Dirac
\usepackage{stmaryrd} % owedge
\SetSymbolFont{stmry}{bold}{U}{stmry}{m}{n} %double brackets

\usepackage{enumitem}

\setlist[enumerate]{leftmargin = 5ex}
\setlist[itemize]{label = {\tiny$\bullet$}, leftmargin = 4ex}

\numberwithin{equation}{section}  

\usepackage{amsthm}
\theoremstyle{plain}
\newtheorem{thm}{Theorem}[section]
\newtheorem{lem}[thm]{Lemma}
\newtheorem{coro}[thm]{Corollary}
\newtheorem{prop}[thm]{Proposition}
\newtheorem{rmk}[thm]{Remark}

\theoremstyle{definition}
\newtheorem{defn}[thm]{Definition}

\newtheorem{examplex}[thm]{Example}
\newenvironment{ex}
  {\pushQED{\qed}\examplex}
  {\popQED\endexamplex}

\newcommand{\bfemph}[1]{\textbf{\textit{#1}}}
\usepackage{xcolor}
\definecolor{shinbashi}{RGB}{53, 143, 183}   % blue

\usepackage[colorlinks = true,
            linkcolor = shinbashi,
            urlcolor  = shinbashi,
            citecolor = shinbashi,
            anchorcolor = shinbashi]{hyperref}

\usepackage{adjustbox}
\usepackage{graphics}
\usepackage{graphicx}
\graphicspath{ {images/} }
\usepackage{tikz-cd}
\usepackage{tikz-3dplot}
\usepackage{caption}
\usepackage{subcaption}
\usepackage{float}

\newcommand{\bdy}{\big|_{x=0}} 
\newcommand{\Bdy}{\Big|_{x=0}}% restricting to boundary x=0

\newcommand{\pd}{\partial}
\newcommand{\nablatt}{\widetilde{\nabla}^\trans}

\newcommand{\Rm}{\mathrm{Rm}}
\newcommand{\Ric}{\mathrm{Ric}}
\newcommand{\scal}{\mathrm{scal}}
\newcommand{\hyp}{\mathrm{hyp}}
\newcommand{\qq}{{(2q)}}
\newcommand{\Ricqq}{{\mathrm{Ric}}^{(2q)}}
\newcommand{\scalqq}{{\mathrm{scal}}^{(2q)}}
\usepackage{xifthen}
\newcommand{\ctr}[1][]{%  contraction
    \mathscr{C}_{\ifthenelse{\isempty{#1}}{g}{#1}}
}

\newcommand{\cA}{\mathcal{A}}

\newcommand{\cC}{\mathcal{C}}

\newcommand{\cF}{\mathcal{F}}
\newcommand{\cG}{\mathcal{G}}

\newcommand{\cI}{\mathcal{I}}

\newcommand{\cK}{\mathcal{K}}
\newcommand{\cL}{\mathcal{L}}

\newcommand{\cQ}{\mathcal{Q}}

\newcommand{\cS}{\mathcal{S}}

\newcommand{\cU}{\mathcal{U}}
\newcommand{\cV}{\mathcal{V}}

\newcommand{\cY}{\mathcal{Y}}

\newcommand{\sG}{\mathscr{G}}

\newcommand{\sI}{\mathscr{I}}
\newcommand{\sO}{\mathscr{O}}

\newcommand{\sS}{\mathscr{S}}
\newcommand{\sV}{\mathscr{V}}

\newcommand{\Ahat}{{\hat{A}}}
\newcommand{\gbar}{{\overline{g}}}
\newcommand{\phibar}{{\overline{\phi}}}
\newcommand{\nablatilde}{\widetilde{\nabla}}

\newcommand{\id}{\operatorname{id}}
\newcommand{\tr}{\operatorname{tr}}
\newcommand{\tf}{\operatorname{tf}}
\newcommand{\Hess}{\operatorname{Hess}}
\newcommand{\ind}{{\operatorname{ind}}} 

\newcommand{\Span}{\operatorname{Span}}

\newcommand{\Diff}{\operatorname{Diff}}
\newcommand{\LimSec}{\operatorname{LimSec}}

\newcommand{\C}{\mathbb{C}}
\newcommand{\N}{\mathbb{N}}
\newcommand{\R}{\mathbb{R}}

\newcommand{\sweight}[1]{\mu_+ + a_{#1}} % sequence in approx solution
\newcommand{\trans}{\mathrm{trans}}

\newcommand{\Adot}{\dot{\cA}}

\newcommand{\zT}{\prescript{0}{}{T}}

\newcommand{\zTX}{\prescript{0}{}{T}X}

\begin{document}

\maketitle
\date{\today}
We study conformally compact metrics satisfying the Lovelock equations, which generalize the Einstein equation. We show that these metrics admit polyhomogeneous expansions, thereby naturally realizing the Fefferman–Graham expansion, which is an important tool in conformal geometry and the AdS/CFT correspondence. In even dimensions, we identify a boundary obstruction to smoothness near the boundary that generalizes the ambient obstruction tensor in the Einstein setting. Under appropriate regularity and curvature conditions, we also construct a formal solution to the singular Yamabe-$(2q)$ problem and provide an index obstruction for the conformally compact Lovelock filling problem of spin manifolds.
\tableofcontents
\pagenumbering{arabic}

%------------------------------
% body
%------------------------------
\section{Introduction} 
The objective of this article is to analyze the boundary asymptotics of conformally compact Lovelock metrics and to investigate two closely related problems: one being the \textit{singular Yamabe–(2q) problem}, and the other the \textit{conformally compact Lovelock (CCL) filling problem} for manifolds with a prescribed conformal infinity. Under suitable regularity and curvature assumptions at the boundary, we show that several foundational results known for conformally compact Einstein metrics extend to this broader class. In particular, we prove the following:
\begin{enumerate}
    \item[(A)] A conformally compact Lovelock metric admits a polyhomogeneous expansion near the boundary, which implies that the formal expansion derived by Albin \cite{Alb20} is indeed realized in the Lovelock setting. 
    \item[(B)] In odd dimensions, a conformal boundary obstruction to the smoothness of the expansion arises; its leading-order term coincides (up to a constant multiple) with the ambient obstruction tensors in the Einstein case.
    \item[(C)] The singular Yamabe-$(2q)$ problem admits formal polyhomogeneous solutions near the boundary.
    \item[(D)] The index of a Dirac-type operator on spin manifolds gives rise to topological obstructions to solving the CCL filling problem under a scalar curvature lower bound.
\end{enumerate}
These results generalize those of conformally compact Einstein metrics from \cite{BH14}, \cite{GH04}, \cite{Gra17} and \cite{GHS21}.

We begin by considering a smooth manifold $X^{n+1}$ with boundary, equipped with a \bfemph{conformally compact metric} $g$ on its interior. This means there exists a defining function $x$, which is a nonnegative function that vanishes simply and exactly on the boundary $M = \pd X$, such that $\gbar = x^2g$ extends to a smooth metric on all of $X$. The \bfemph{conformal boundary} of~$X$ is the pair $(M, [\gbar|_{TM}])$, and the induced conformal class $[h] = [\gbar|_{TM}]$ on $M$ is referred to as the \bfemph{conformal infinity} of $g$. A paradigmatic example is the hyperbolic space $H^{n+1}$, whose conformal boundary is the standard round sphere $S^n$.

We are particularly interested in conformally compact metrics that satisfy a natural generalization of the Einstein equation, known as the \bfemph{Lovelock equation}. This equation arises as the Euler–Lagrange equation associated with the Lovelock action, which involves higher-order curvature invariants, and it takes the form
\begin{equation} \label{intro eqn: F_g}
    F_{\alpha} (g) = \sum_q \alpha_q \Big(\Ric_g^\qq - \lambda^\qq g \Big) -\frac{\alpha_q}{2q} \Big( \scal^\qq_g - (n+1) \lambda^\qq \Big)g = 0,
\end{equation}
for some $q$-tuple $\alpha = (\alpha_1, \alpha_2, \cdots, \alpha_q)$, where 
\begin{equation*}
    \Big(\Ricqq_g\Big)_i^j = \delta^{i i_1 \cdots i_{2q}}_{j j_1 \cdots i_{2q}} R_{i_1 i_2}^{j_1 j_2} \cdots R_{i_{2q-1} i_{2q}}^{j_{2q-1} j_{2q}}, \quad \scalqq_g = \tr(\Ricqq_g ), \quad \delta_{j_1 j_2 \cdots j_{2q}}^{i_1 i_2 \cdots i_{2q}} = \det\Big((\delta^{i_s}_{j_t})\Big),
\end{equation*}
and $\lambda^\qq = \left( -\frac{1}{2}\right)^q \frac{n! (2q)!}{(n-2q+1)!}$ are the normalizing constants chosen so that the Lovelock equation is satisfied by the hyperbolic metric. A conformal compact metric $g$ satisfying the Lovelock equation is referred to as a \bfemph{conformally compact Lovelock metric}. 

Albin \cite{Alb20} proved that such metrics exist on the ball $B^{n+1}$ when the conformal infinity~$[h]$ is sufficiently close to the round metric on $S^n$. This result illustrates an idea in geometric analysis: one starts with a conformal class $[h]$ on the boundary and seeks to construct a bulk metric $g$ on $X$ which realizes $[h]$ as the prescribed conformal infinity.

A related perspective arises in the AdS/CFT correspondence of theoretical physics \cite{Mal98, Wit98}, where anti-de Sitter (AdS) spaces are modeled as conformally compact manifolds, and their conformal boundaries encode the data of conformal field theories (CFTs). Mathematically, this correspondence is captured by the Fefferman–Graham ambient metric construction \cite{FG85, FG12}, which associates to a conformal manifold $(M, [h])$ an ambient space~$(X^{n+1}, g)$ compatible with the prescribed conformal infinity. This construction provides a framework for analyzing conformal invariants of $[h]$ through the more accessible Riemannian invariants of $(X, g)$.

In recent work, Albin \cite{Alb20} demonstrated that conformally compact Lovelock metrics admit a formal \bfemph{Fefferman–Graham expansion}, thereby extending the classical expansion theory developed for Einstein metrics to the Lovelock setting. Near the boundary, the metric $g$ takes the form 
\[x^2 g = dx^2 + h_x,\]
where $h_x$ admits the asymptotic expansion 
\begin{align} \label{intro eqn: Fefferman-Graham expansion}
    h_x = \begin{cases}
        h_0 + h_2 x^2 + \text{(even powers)} + h_{n-1} x^{n-1} + h_n x^n + \cdots & n \text{ odd}\\
        h_0 + h_2 x^2 + \text{(even powers)} + h_{n,1} \log (x)x^{n-1} + h_n x^n + \cdots  & n \text{ even}.
    \end{cases}
\end{align}
Our first aim is to show that conformally compact Lovelock metrics exhibit polyhomogeneous behavior near the boundary. This ensures that the formal Fefferman–Graham expansion constructed by Albin \cite{Alb20} is not merely formal but is genuinely realized in the Lovelock setting. In the conformally compact Einstein case, polyhomogeneity is a well-established property, proven independently by Chruściel, Delay, Lee, and Skinner \cite{CDLS05}, and by Biquard and Herzlich \cite{BH14} through gauge fixing and the application of elliptic theory for Laplace-type operators. We adapt the latter framework to the Lovelock setting, thereby extending the regularity theory from the Einstein case to a broader class of higher-curvature gravity models.
\begin{thm}[Polyhomogeneity of Asymptotically Hyperbolic Lovelock Metrics]
    Let $X$ be a manifold with boundary equipped with an asymptotically hyperbolic Lovelock metric $g$, which induces a smooth conformal infinity $[h]$. Then $g$ is polyhomogeneous in a neighborhood of the boundary.
\end{thm}
When $n$ is even, the Fefferman–Graham expansion of $h_x$ may fail to be smooth due to the appearance of a logarithmic term. This obstruction was studied in the Einstein setting by Graham and Hirachi \cite{GH04}, who described its leading-order structure in terms of the ambient obstruction tensor. We generalize their result to the Lovelock case and demonstrate that the leading-order term of the obstruction tensor remains unchanged up to a dimension-dependent constant factor.

\begin{thm}[Leading-Order Term of the Lovelock Obstruction Tensor] \label{intro thm: obstruction leading order}
    Let $M$ be a smooth $n$-manifold with conformal infinity $[h]$, where $n \geq 4$, and let $X$ be an $(n+1)$-manifold with boundary $M$. Let~$x$ denote a boundary defining function for $M$.

    There exists a metric $g$ satisfying $x^2 g|_{TM} \in [h]$ and $F_\alpha(g) = O(x^{n-2})$. 
    The leading-order term of the obstruction tensor $\sO = x^{2-n} \tf\big(F_\alpha(g)\big)\big|_{x=0}$ is given by
    \begin{equation} \label{intro eqn: h.o.t. of O_ij}
        \frac{A_1(\alpha)}{c_n} \Delta^{\frac{n}{2} - 2} \left( \tensor{P}{_{ij,k}^k} - \tensor{P}{_k^k_{,ij}} \right),
    \end{equation}
    where $P$ is the Schouten tensor of $h$, and the coefficients are given by
    \begin{align*}
        A_1(\alpha) = \sum_q \alpha_q \left( -\frac{1}{2} \right)^{q-1} \frac{(n-2)!}{2} \frac{(2q-1)!}{(n-2q)!}, \quad \text{and} \quad c_n = \frac{2^{n-2}(n/2-1)!^2}{n-2}.
    \end{align*} 
\end{thm}

The singular Yamabe problem, introduced in \cite{LN74} and further developed in \cite{Gra17}, seeks a defining function $u$ for the boundary $M = \partial X$ of a smooth compact Riemannian manifold $(X^{n+1}, \bar{g})$ such that the conformally rescaled metric $g = u^{-2} \bar{g}$ has constant scalar curvature. This framework extends naturally to conformally compact Lovelock metrics, wherein the scalar curvature is replaced by scalar-$(2q)$ curvatures. As in the Einstein case, we study the boundary asymptotics of defining functions $u$  that yield metrics with constant scalar-$(2q)$ curvatures to be constant. These solutions admit formal polyhomogeneous expansions near the boundary, analogous to those arising in the classical singular Yamabe problem.

For a $q$-tuple $\beta = (\beta_1, \ldots, \beta_q)$ with $2q \leq n+1$ and $\kappa > 0$, we set
\[\widetilde{B}_{1,2}(\beta, \kappa) = \sum_q \beta_q \left( -\frac{\kappa}{2} \right)^{q-1} \frac{(n-1)!}{2} \frac{(2q)! \;q}{(n-2q+1)!}.\]
\begin{thm}[Formal Polyhomogeneous Solution to the Singular Yamabe-$(2q)$ Problem]
    For a prescribed $\beta = (\beta_1, \cdots, \beta_q)$, with $2q\leq n$, satisfying $\widetilde{B}_{1,2}(\beta, \kappa) \neq 0$, there exists a function $u$ on $X$ that serves as a boundary defining function of $M = \pd X$ such that the scalar-$\qq$ curvatures of the conformally rescaled metric $g = u^{-2}\gbar$ satisfy
    \begin{equation}
        \sum \beta_q \Big(\scal_g^\qq - (n+1)\lambda^\qq\Big) = O(x^{n+2} \log x),
    \end{equation}
    where $\lambda^\qq = \big(-\frac{1}{2}\big)^q \frac{n!(2q)!}{(n-2q+1)!}$. 
    Moreover, the function $u$ has the form
    \[u = x + u_2 x^2 + \cdots + u_{n+1} \,x^{n+1} + \cL^\qq x^{n+2} \log x, \]
    where $\cL^\qq$ is the singular Yamabe-$(2q)$ obstruction.
\end{thm}

Finally, we address obstructions to the \bfemph{conformally compact Lovelock (CCL) filling problem}, which generalizes the study of conformally compact Einstein (CCE) filling problem explored by \cite{GL91, GHS21, CG22}. The CCL filling problem asks whether, given a conformal class $[h]$ of metrics on the boundary $M  = \pd X$, there exists a conformally compact Lovelock metric $g$ on the interior of $X^{n+1}$ whose conformal infinity matches the given class. While examples exist where such fillings are possible (e.g., filling the ball with prescribed conformal infinity near that of the round sphere metric), we show that certain conformal classes cannot be realized as the conformal infinity of a conformally compact Lovelock metric. When the dimension satisfies $n+1 = 4k$, with $k \geq 2$, we demonstrate that the obstruction given in \cite{GHS21}, based on the index of the Dirac operator, also applies to the CCL-filling problem, when a lower bound on the scalar curvature $\scal_g$ is assumed.

\begin{thm}[Obstruction to CCL Fillings] \label{intro thm: CCL filling obstruction}
    Let $X$ be a smooth, compact Riemannian spin manifold with boundary, satisfying $\dim X = n+1 = 4k$ with $k \geq 2$. Suppose that $h$ is a Riemannian metric on the boundary $M = \pd X$. If the following conditions hold:
    \begin{enumerate} 
        \item the boundary metric $h$ admits a positive Yamabe invariant, 
        \item the Dirac operator associated with $g$ has non vanishing index $\cI$,
    \end{enumerate} 
    then the pair $(X, h)$ does not admit any conformally compact Lovelock filling $(X, g)$ that satisfies $\scal_g \geq -n(n+1)$.
\end{thm}

The structure of this paper is as follows: In \S\ref{sec: conformal geometry} and \S\ref{sec: Lovelock tensor}, we present the foundational concepts of asymptotically hyperbolic Lovelock metrics. Specifically, \S\ref{sec: conformal geometry} introduces 0-differential operators, which are beneficial in conformal geometry, and defines the Hölder and polyhomogeneous spaces that we will be working with. In \S\ref{sec: Lovelock tensor}, we define the Lovelock tensor using double forms and the contraction map. We also discuss DeTurck's trick, which enables the application of elliptic theory to the Lovelock equations.

In \S\ref{sec: invertibility of Laplace operators}, we focus on the invertibility of Laplace-type operators that arise in the modified Lovelock tensor after applying DeTurck's trick. The proof of polyhomogeneity is given in \S\ref{sec: phg}, using the techniques developed in \cite{BH14}. In cases where logarithmic terms appear in the Fefferman-Graham expansion, we consider the obstruction tensor. The computation of its leading-order term is presented in \S\ref{sec: ambient obstruction}, following the methods outlined in \cite{GH04}.

Moving forward, we address the singular Yamabe-$(2q)$ problem and the CCL filling problem. A formal expansion of the solution to the singular Yamabe-$(2q)$ problem is provided in \S\ref{sec: singular Yamabe}, while \S\ref{sec: CCL filling} discusses the index obstruction to a CCL filling for a spin manifold.
\\
\\
\textbf{Acknowledgment}. I am deeply grateful to my advisor, Pierre Albin, for introducing me to this problem and for guiding me throughout this work. I also thank Gabriele La Nave for many enlightening discussions on modified gravity and gravitational instantons, which have greatly enriched my understanding. I have also benefited from inspiring discussions with Hadrian Quan, Gayana Jayasinghe, and Karthik Vasu, as well as from cheerful conversations with Alex Taylor and Huy Tran. 

\section{Conformal geometry} \label{sec: conformal geometry}
We will consider a compact $(n+1)$-dimensional manifold $X$ with boundary $\pd X = M$, where $n \geq 4$. We assume the interior $\mathring{X}$ is equipped with a metric $g$. A \bfemph{boundary defining function} $x$ is a smooth function that is positive on $\mathring{X}$, vanishes only on $M$ and has a nonvanishing differential on $M$. We identify a collar neighborhood of the form $\{ x < x_0\}$ of $M$ in $X$ with a collar of $M$, $[0,x_0] \times M$, and denote it $M_{\leq x_0}$. We call a boundary defining function $x$ \bfemph{special} if $\left|\frac{dx}{x}\right|^2_g \equiv \kappa$ in a neighborhood of $M$, for some constant $\kappa$.

A \bfemph{conformally compact metric} is a metric $g$ on $\mathring{X}$ such that $\bar{g} = x^2 g$ extends continuously to a metric on $X$. We denote the restriction $\bar{g}|_{TM} = h$. It is clear that the conformal class $[h]$ is independent of the choice of $x$. We call it the \bfemph{conformal infinity}. We say that such a metric $g$ is \bfemph{of class $\boldsymbol{C^{k,\gamma}}$} for a nonnegative integer $k$ and $0 < \gamma < 1$ if for any $x$ as above, $x^2g$ extends continuously to a metric of class $C^{k,\gamma}$. We will assume the compactified metric $\bar{g}$ is of class $C^2$ and the metric $h$ is smooth. 

\begin{rmk}
    As a convention, Greek letters (e.g., $\mu, \nu$) run from $0$ to $n$, labeling the coordinates on the manifold with boundary $X$; whereas Latin letters (e.g., $i, j$) run from $1$ to $n$, labeling the coordinates on the boundary $M$. We denote $\pd_0 = \pd_x$.
\end{rmk}

\subsection{0-differential operators}
The type of question we are interested in is the asymptotic behavior of conformally compact Lovelock metrics. We will discuss Lovelock metrics in \S\ref{sec: Lovelock tensor}. When studying conformally compact  metrics, it is convenient to use the framework of the 0-calculus. Our main focus will be introducing the $0$-differential operators. For a comprehensive discussion and an examination of elliptic theory in the 0-calculus setting, we refer the reader to \cite{Maz91}. 

Suppose $X$ is a manifold with boundary $\pd X = M$ which is identified by the boundary defining function $x$. The space of all smooth sections of $TX$, $\cV(X) = C^\infty(X;TX)$ is the Lie algebra of smooth vector fields on $X$. It has a Lie subalgebra 
\[\cV_0(X) = \{V \in \cV(X): V \text{ vanishing on } M\}.\]
There is a bundle $\zTX$ with a bundle map $\iota_0: \zTX \to TX$ which induces a map of sections 
\[\iota_0^*: C^\infty(X; \zTX) \to C^\infty(X; TX)\]
whose image is $\cV_0(X)$ \cite[\S2]{Maz91}. If $\cI_x \subset C^\infty(X)$ is the ideal of functions vanishing at~$x \in X$, then the fiber of $\zTX$ is
\[\zT_xX = \cV_0(X) \Big/ \cI_x\cdot\cV_0(X).\]
The map $\iota_0$ is an isomorphism in the interior $\mathring{X}$, since $C^\infty(\mathring{X}; \zTX)$ coincides with the space~$C^\infty(\mathring{X}; TX)$, but is neither injective nor surjective when restricting to the boundary of~$X$. We call $\zTX$ the \bfemph{0-tangent bundle}. Its dual bundle is the \bfemph{0-cotangent bundle}~ $\zT^*X$.
The differential operators generated by $\cV_0(X)$ are called \bfemph{0-differential operators} on $X$.

For our purposes, we consider an asymptotically hyperbolic metric of the form
\[g = \frac{dx^2 + h_x}{x^2},\]
where $h_x$ is a family of metrics on $M$ parametrized by $x$, and we denote its restriction by~$h_0 = h_x|_{TM}$. We can express $\cV_0$ and $\zT^*X$ in terms of local coordinates $(x,y^1, \cdots, y^n)$, where $x$ is the boundary defining function and $y^i$ restricts to coordinates on $M$. The norm of the vector fields $\pd_x, \pd_{y^i}$ blow up as $x \to 0$, whereas the length of $x\pd_x, x\pd_{y^i}$ is pointwise bounded above. The Lie subalgebra we discussed above is 
\[\cV_0 = \Span_{C^\infty(X)}\{x\pd_x, x \pd_{y^i}\},\]
and the 0-differential operators of order $k$ are of the form 
\[\sum_{j+|\alpha|\leq k} b_{j,\alpha}(x,y)(x\pd_x)^j (x\pd_y)^\alpha,\quad  b_{j,\alpha}(x,y) \in C^\infty(X).\]
We denote \bfemph{the space of 0-differential operators of order $\boldsymbol{k}$} by $\Diff_0^k(X)$.
The 0-cotangent bundle $\zT^*X$, is locally spanned by the set of zero 1-forms 
\[\Big\{ \frac{dx}{x}, \frac{dy^i}{x} \Big\}.\]
We emphasize that $\frac{dx}{x}$ and $\frac{dy^i}{x}$ are smooth as sections of $\zT^*X$.

\subsection{Polyhomogeneity}
We introduce the spaces of polyhomogeneous sections which generalize smooth sections by allowing logarithmic dependence and noninteger powers of $x$ in their asymptotic expansions. These spaces are crucial for the discussion in \S\ref{sec: phg}, where we define the Green's operator for Laplace type operators. 

Given a manifold $X$ with boundary and a boundary-defining function $x$, the \bfemph{single-weighted H\"older space} on $X$ is defined by
\begin{equation} \label{eqn: single-weighted Holder space}
    C^{k,\delta}_\xi(X) = x^\xi C^{k,\delta}(X).
\end{equation}

For the purpose of local analysis, we now focus on a semi-ball and define the double-weighted Hölder spaces following \cite[\S4]{BH14}. 
\begin{rmk}
    We will use $\phi, \phi_0$, etc., to denote metrics on the semi-ball, and we will use~$g, g_0$, etc., to denote metrics on the manifold with boundary $X$.
\end{rmk}

Let $B^{n+1}_+$ be the $(n+1)$-dimensional semi-ball $\{x_0^2 + x_1^2 + \cdots + x_n^2 \leq 1, x_0 \geq 0\}$. Its boundary at infinity has a boundary defining function $x = x_0$. Suppose $h$ is a smooth metric on $M = \{x = 0\}$, we consider the metric on $B^{n+1}_+$ given by
\[\phi_0 = \frac{dx^2+h}{x^2}.\]
When $h = dx_1^2 + \cdots + dx_n^2$, $\phi_0$ is the hyperbolic metric and we denote it by $\phi_{\hyp}$. 

The semi-ball has two boundaries. The boundary at infinity is given by $M = \{x_0=0\}$. We denote the boundary defining function of $M$ by $x = x_0$. The other one is the boundary in the interior given by $\pd_{\mathrm{int}} B_+^{n+1} = \{x_0^2 + x_1^2 + \cdots + x_n^2 = 1\}$. The intersection of these two boundaries is the sphere $S^{n-1} = \{x_0 = 0, x_1^2 + \cdots + x_n^2 = 1\}$. 

A double-weighted H\"older space is similar to the single-weighted H\"older space but it captures the behavior of functions near the two distinct boundaries. We set $\rho$ as the distance from a point in $M$ to the sphere $S^{n-1}$. Then the metric $h$ decomposes as 
\[h = d\rho^2 + h_\rho,\]
where $h_\rho$ is a family of metrics on $S^{n-1}$. We define $u$ and $v$ by
\[x = \frac{2u}{1+u^2} v, \quad \rho = \frac{1-u^2}{1+u^2} v,\]
then 
\[\phi_0 = \frac{du^2+\widetilde{h}}{u^2}, \quad \text{where} \quad \widetilde{h} = \frac{h}{4\rho^2} = \frac{(1+u^2)^2}{4v^2}(dv^2 + h_\rho).\]

\begin{center}
    \includegraphics[width=0.35\textwidth]{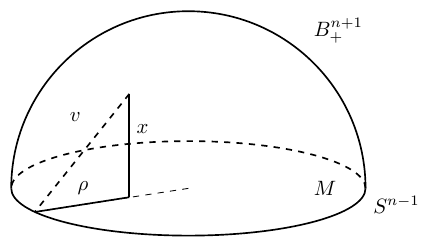}
\end{center}

We define the \bfemph{double-weighted H\"older space} to be
\[C^{k,\delta}_{\xi, \zeta}(B_+^{n+1}) = u^{\xi} v^{\zeta} C^{k,\delta}(B_+^{n+1}).\]
\begin{rmk}
    The coordinates $(u,v)$ are related to distances to the boundaries. Note that we have $v = \sqrt{x^2+\rho^2}$ representing the distance of a point $p \in B_+^{n+1}$ to $S^{n-1}$. In particular, for points on the boundary $M$, $x = u = 0$, and $v = \rho$. Also, $u = 1$ implies $x = v$ and $\rho = 0$, which represent points on the sphere $S^{n-1}$.
\end{rmk}

We now define polyhomogeneous functions on a semi-ball following \cite[\S6]{BH14}. 
\begin{defn}
    An \bfemph{index set} is a subset $E \subset \C \times \N_0$ satisfying 
    \begin{itemize}
        \item for each $s \in \R$, $E_{\leq s} = \{(z,k) \in E: \operatorname{Re}(z)\leq s\}$ is finite;
        \item if $(z,k) \in E, 0\leq j \leq k$ then $(z,j) \in E$.
    \end{itemize}
\end{defn}
Throughout this article, we will consider index sets $E$ that are subsets of $\R \times \N_0$. 

\begin{defn}
    Let $\xi \geq 0$ and $\zeta > 0$. We define the space $\Adot_{\xi,\zeta}(B_+^{n+1})$ of \bfemph{finite polyhomogeneous functions of weight at least $\boldsymbol{\xi}$} to be the collection of functions of the form
    \begin{align} \label{eqn: phg expansion 1}
        \sum_{(z,k)\in E} a_{z,k}(y) \, x^z (\log x)^k,
    \end{align}
    where $E$ is a finite subset of $[\xi,\infty) \times \N_0$ and $a_{z,k}(y) \in \rho^\zeta C^{k,\gamma}(M)$. 
\end{defn}

In addition, we define $\Adot[\xi]_\zeta(B_+^{n+1})$ to be the set of finite polyhomogeneous functions in~$\Adot_{\xi,\zeta}(B_+^{n+1})$ of the form 
\[\sum_{k} a_{z,k}(y) \, x^\xi \, (\log x)^k.\]
If $u \in \Adot_{\xi,\zeta}(B_+^{n+1})$ admits the expansion \eqref{eqn: phg expansion 1}, we denote by $[u]_\xi$ the sum of the terms in the expansion \eqref{eqn: phg expansion 1} with $z = \xi$. Then $[u]_\xi \in \Adot[\xi]_\zeta(B_+^{n+1})$.

\begin{defn}
    A \bfemph{polyhomogeneous function} $u$ on a semi-ball with respect to the weights $(\xi, \zeta)$ is a function for which there exist coefficients $a_{z,k}(y) \in \rho^\zeta C^\infty(M)$, where~$(z,k)$ belongs to an index set $E \subset [\xi, \infty) \times \N_0$, such that for all $\delta > 0$, there exists an integer $N$ for which the following holds:
    \[u(x,y) - \sum_{(z,k) \in E_{\leq N}} a_{z,k}(y) \, x^z (\log x)^k \; \in \; C^\infty_{\delta,\zeta}(B_+^{n+1}).\]
    We denote the set of polyhomogeneous functions as $\cA_{\xi, \zeta}(B^{n+1}_+)$. 
\end{defn}

\begin{rmk}
    The definitions of weighted H\"older functions extend naturally to sections of tensor bundles. We say that a tensor is $C^{k,\delta}_\xi$ (or $C^{k,\delta}_{\xi,\zeta}$, respectively) if, within trivializations that contain boundary points, each of its components belongs to $C^{k,\delta}_\xi$ (or $C^{k,\delta}_{\xi,\zeta}$, respectively). Similarly, the definitions of polyhomogeneous functions extend to sections of natural bundles, with the additional requirement that the coefficients $a_{z,k}$ are parallel along the vector field~$x\pd_x$ with respect to the Levi-Civita connection \cite[p.831]{BH14}.
\end{rmk}

\subsection{Operator algebra} \label{sec: operator algebra Q}
Let $B^{n+1}_+$ be the semi-ball defined previously. Consider a collar neighborhood $M_{\leq x_0}$ as a collar neighborhood of the infinity boundary. Suppose $g_0$ is a metric on a $B^{n+1}_+$ of the form $\frac{dx^2+h}{x^2}$, where $h$ is a smooth metric on $M$. We choose an extension $\nablatilde$ of the Levi-Civita connection of $h$ to the interior of $B^{n+1}_+$ such that
\[\widetilde{\nabla}(x \partial_x) = 0, \quad \nablatilde_{x \partial_x}(x \pd_y) = 0.\]  
The following two algebras of differential operators generated by $\nablatilde$ will be convenient to use in the proof of polyhomogeneity: we denote by $\cQ_1$ the subalgebra of differential operators on the space of symmetric 2-tensors generated by 
\[ \nablatilde_{x \pd_x}, \nablatilde_{x\pd_y} \text{ and polyhomogeneous functions in } \Adot_{0,0};\]
and denote by $\cQ_0$ the subalgebra of $\cQ_1$ generated by 
\[ \nablatilde_{x\pd_y} \text{ and polyhomogeneous functions in } \Adot_{1,0};\]
and let $\cQ = \cQ_0 + x\cQ_1$.

\section{Lovelock metrics} \label{sec: Lovelock tensor}
In this section, we introduce Lovelock metrics, a class of Riemannian metrics that generalize the well-known Einstein metrics by incorporating a nonlinear dependence on the second-order derivatives of the metric tensor. 
We begin with the formalism of double forms, a natural framework in which classical geometric notions such as metrics, curvatures, and curvature contractions find their natural realization. This formalism was first introduced by Kulkarni in \cite{Kul72} and was later extensively developed by Labbi in \cite{Lab05, Lab07, Lab08, Lab08-2}. Throughout this section, we consider a smooth Riemannian manifold $(X,g)$ of dimension~$m = n+1$.

\subsection{Double forms} 
Let $\Lambda^p T^*X \otimes \Lambda^q T^*X$, with $p,q \geq 0$ be the tensor product of exterior powers of the cotangent bundle. The \bfemph{ring of differential double forms} is the direct sum~$\Omega^{\bullet \otimes \bullet }(X) = \bigoplus_{p,q} \Omega^{p\otimes q}(X)$, where
\[\Omega^{p\otimes q}(X) = \cC^\infty (X;\Lambda^p T^*X \otimes \Lambda^q T^*X)\]
and it is a bi-graded associative ring. An element $\omega \in \Omega^{p\otimes q}(X)$ is called a \bfemph{$\boldsymbol{(p,q)}$-double form}. 

\begin{ex}
    A Riemannian metric $g = g_{ij} \,dx^i \otimes dx^j$ can be viewed as a double form in~$\Omega^{1\otimes 1}(X)$. 
    
    The Riemannian curvature tensor $R$ is a particularly interesting example. The algebraic symmetries observed in the components $R_{ijkl}$ suggest a natural way to identify the $(4,0)$-tensor $R$ with a $(2,2)$-double form $\Rm$. In local coordinates, we write
    \[\Rm = R_{ijkl} \,(dx^i \wedge dx^j) \otimes (dx^k \wedge dx^\ell).\]
    Similarly, we can identify the Ricci tensor $\Ric$, and the Schouten tensor $P$ as $(1,1)$-forms and the Weyl tensor $W$ as a $(2,2)$-form.
\end{ex}

Let $\cC^p(X)$ denote the space of symmetric $(p,p)$-double forms, defined as
\[
\cC^p(X) := \left\{ \omega \in \Omega^{p \otimes p}(X) \mid \omega(\xi, \eta) = \omega(\eta, \xi) \text{ for all } \xi, \eta \in \Lambda^p T^*X \right\}.
\]
The graded space $\cC(X) := \bigoplus_{p \geq 0} \cC^p(X)$ is what Kulkarni referred to in \cite[\S1]{Kul72} as the \bfemph{ring of curvature structures}. This terminology reflects the fact that the Riemann curvature tensor and its powers are naturally represented as elements of this ring. To avoid confusion with function spaces, we will use $\Omega^{p\otimes p}$ instead of $\cC^p$ to denote the curvature structures.

\subsection{The Kulkarni-Nomizu product and the contraction map}
Multiplication in the ring $\Omega^{\bullet \otimes \bullet}(X)$ is given by the \bfemph{Kulkarni-Nomizu product}
\begin{align*}
    .: \Omega^{p\otimes q}(X) \times \Omega^{r\otimes s}(X)  &\to \Omega^{(p+r)\otimes (q+s)}(X), \\
    (\omega, \eta) &\mapsto \omega. \eta 
\end{align*}
For $\omega \in \Omega^{p\otimes q}(X)$, $\eta \in \Omega^{r\otimes s}(X)$, we have
which is defined for $\omega = \alpha_1 \otimes \alpha_2$ and $\eta = \beta_1 \otimes \beta_2$ by 
\[\omega. \eta = (\alpha_1 \otimes \alpha_2) . (\beta_1 \otimes \beta_2) = (\alpha_1 \wedge \beta_1) \otimes (\alpha_1 \wedge \beta_2),\] 
and then extended to all double forms by linearity.
Alternatively, identifying a $(p,q)$-double form with a family of bilinear forms $\Lambda^p TX \otimes \Lambda^q TX \to \mathbb{R}$, the product is given by
\begin{alignat*}{2}
    \omega. \eta (V_1 &\wedge \cdots \wedge V_{p+r}; W_1 \wedge \cdots \wedge W_{q+s})
    &&  \\
    & = \frac{1}{p!q!r!s!} \sum_{\sigma \in S_{p+r}, \tau \in S_{q+s}} \operatorname{sgn}(\sigma \tau) \;  
    &&\omega (V_{\sigma(1)}\wedge \cdots \wedge V_{\sigma(p)} ; W_{\tau(1)} \wedge \cdots W_{\tau(q)})\\
    & 
    &&\eta (V_{\sigma(p+1)} \wedge \cdots \wedge V_{\sigma(p+r)} ; W_{\tau(q+1)} \wedge \cdots W_{\tau(q+s)}).
\end{alignat*}
In practice, we often omit the dot, and $\omega.\omega$ is denoted as $\omega^2$. 

\begin{ex} 
    The product of two (1,1)-forms $\omega$ and $\eta$ is given by 
    \[\omega. \eta = (\omega_{ik}\eta_{jl} + \omega_{jl}\eta_{ik} - \omega_{il}\eta_{jk} - \omega_{jk}\eta_{il}) \; \big((dx^i \wedge dx^j) \otimes (dx^k \wedge dx^\ell)\big).\]
    In particular, when $\omega = \eta = g$, $\omega. \eta = g^2 = 2(g_{ik}g_{jl} - g_{il}g_{jk}) \; \big((dx^i \wedge dx^j) \otimes (dx^k \wedge dx^\ell)\big)$.
\end{ex}
The \bfemph{contraction} is a map 
\[\ctr: \Omega^{p\otimes q}(X) \to \Omega^{(p-1)\otimes (q-1)}(X) \]
For $\omega \in \Omega^{p\otimes q}(X)$, and $\{e_i\}$ orthonormal basis of vector fields on $X$, we define
\[\ctr (\omega) (V_1 \wedge \cdots \wedge V_{p-1}; W_1 \wedge \cdots \wedge W_{q-1}) = \sum_{i=1}^m \omega (e_i \wedge V_1 \wedge \cdots \wedge V_{p-1}; e_i \wedge W_1 \wedge \cdots \wedge W_{q-1}),\]
with the convention that, if $p = 0$ or $q = 0$, we set $\ctr \omega = 0$. 
\begin{ex}
    The Ricci curvature and scalar curvature are contractions of $\Rm_g$: 
    \[\Ric_g = \ctr(\Rm_g) \quad \text{and} \quad \scal_g = \ctr^2(\Rm_g).\]
\end{ex}

When the double form is of the form $g^k \omega$, the contraction~$\ctr^\ell(g^k \omega)$ admits an expression in terms of~$g$ and~$\ctr^i(\omega)$ for $i \leq k$ \cite[Lemma 1.1]{Alb20}.
\begin{lem}[Contraction Formula]
    If $\omega \in \Omega^{p\otimes q}(X)$, then 
    \begin{equation} \label{eqn: contraction formula}
        \ctr^\ell(g^k \omega) = \sum_{r=0}^\ell \binom{m-p-q+\ell-k}{r} \frac{k!}{(k-r)!} \frac{\ell!}{(\ell-r)!} \;g^{k-r} \;\ctr^{\ell-r}(\omega),
    \end{equation}
    with the conventions that $g^{k - r} = 0$ if $k - r < 0$, and $\ctr^{\ell - r}(\omega) = 0$ if $\ell - r \notin [0, \min\{p, q\}]$.
\end{lem}
We refer to Equation~\eqref{eqn: contraction formula} as the \bfemph{contraction formula}. The following special cases and a corollary will be used in Sections~\ref{sec: ambient obstruction} and~\ref{sec: singular Yamabe}.

\begin{prop} \label{prop: -1 contraction of Rm and g}
    Applying the contraction formula, we have for an $(1,1)$-double form $\eta$ and a $(2,2)$-double form $\omega$
    \begin{align*}
        \ctr^{\ell} \Big( g^\ell \eta \Big) &= \frac{(m-2)!\,\ell!}{(m-\ell-1)!} \; \Big[ (m-\ell-1) \eta + \ell \, \ctr(\eta) \,g \Big],\\
        \ctr^{\ell} \Big( g^{\ell-1} \omega \Big) &= \frac{(m-3)!\,\ell!}{(m-\ell-1)!} \; \Big[ (m-\ell-1) \ctr(\omega) + \frac{\ell-1}{2} \ctr^2(\omega)\, g \Big].
    \end{align*}
\end{prop}
\begin{prop} \label{prop: full contraction l form times 1 form}
    For an $(\ell,\ell)$-double form $\tau$ and an $(1,1)$-double form $\eta$, we have
    \[\ctr^{\ell+1} (\tau \cdot \eta) = -(\ell+1)(\ell-1) \, \ctr^{\ell}(\tau) \, \ctr(\eta) + \frac{1}{2}(\ell+1) \ell \; \ctr^2 \Big( \ctr^{\ell-1}(\tau) \cdot \eta \Big). \]
\end{prop}
\begin{proof}
    We begin by observing that the left-hand side is a full contraction. The contraction of $\tau \cdot \eta$ can be expressed as a linear combination of terms in which indices of $\tau$ and $\eta$ are paired with other indices of $\tau$ and $\eta$ using the metric. Since $\eta$ is a $(1,1)$-form, in terms of $\eta$ there will be two types of terms:

    \begin{itemize}
        \item When the indices of $\eta$ are summed over, this corresponds to $\ctr^{\ell}(\tau) g^{ij} \eta_{i,j}$,
        \item Alternatively, when the indices of $\eta$ are paired with other indices of $\tau$, this corresponds to
        $\Big[\ctr^{\ell-1}(\tau)\Big]_{a,b} g^{ja} g^{ib} \eta_{i,j}$.
    \end{itemize}  
    Since the contraction is purely combinatorial, we can express the result as a linear combination of the two terms above:
    \[\ctr^{\ell+1} (\tau \cdot \eta) = A_\ell \, \ctr^{\ell}(\tau) \, \ctr(\eta) + B_\ell \, \ctr^2 \Big( \ctr^{\ell-1}(\tau) \cdot \eta \Big), \]
    where the constants $A_\ell$ and $B_\ell$ are independent of $\tau$ and $\eta$. To determine these constants, consider $\tau = g^{\ell-1}\omega$ for some arbitrary $\omega, \eta \in \Omega^{1\otimes 1}$. Applying the contraction formula, we obtain:
    \begin{align} \label{eqn: full ctr prop}
    \begin{split}
        \frac{1}{2} (m-\ell)(\ell+1)\ell \,\ctr^2(\omega \eta) =& (m-1)\ell \,A_\ell \,\ctr(\omega) \ctr(\eta) \\
        &+ (m-\ell) B_\ell \, \ctr^2(\omega\eta) +  2 (m-1) (\ell-1) B_\ell \,\ctr(\omega) \ctr (\eta).
    \end{split}
    \end{align}
    Note that the term $g^{ij}g^{ab} \omega_{i,b} \eta_{j,a}$ only appears in $\ctr^2(\omega\eta)$. To satisfy Equation \eqref{eqn: full ctr prop}, we must have:
    \[\frac{1}{2} (m-\ell)(\ell+1)\ell \,\ctr^2(\omega \eta) = (m-\ell) B_\ell \, \ctr^2(\omega\eta).\]
    Solving for $B_\ell$, we find that $B_\ell = \frac{1}{2}(\ell+1) \ell$. Substituting this into Equation \eqref{eqn: full ctr prop}, we find that $A_\ell = - (\ell+1)(\ell-1)$.
\end{proof}

\subsection{Lovelock metrics and Lovelock tensors}
Lovelock geometry extends Einstein's theory of gravity by incorporating nonlinear dependencies on Riemannian curvature \cite{Lov71}. When seeking symmetric, divergence-free tensors that are polynomials in the Riemannian metric $g$ and its first two derivatives, the solution space is spanned by $g$ and
\begin{equation} \label{eqn: generalized Ricci curvature}
    \delta^{i i_1 \cdots i_{2q}}_{j j_1 \cdots i_{2q}} R_{i_1 i_2}^{j_1 j_2} \cdots R_{i_{2q-1} i_{2q}}^{j_{2q-1} j_{2q}}, 
\end{equation}
where $\delta^{i i_1 \cdots i_{2q}}_{j j_1 \cdots i_{2q}}$ denotes the generalized Kronecker delta 
\[\delta^{i_1 \cdots i_k}_{j_1 \cdots j_k} = \det\begin{pmatrix}
    \delta^{i_1}_{j_1} & \delta^{i_1}_{j_2} & \cdots & \delta^{i_1}_{j_k}\\
    \delta^{i_2}_{j_1} & \delta^{i_2}_{j_2} & \cdots & \delta^{i_2}_{j_k}\\
    \vdots & \vdots & \ddots & \vdots\\
    \delta^{i_k}_{j_1} & \delta^{i_k}_{j_2} & \cdots & \delta^{i_k}_{j_k}
\end{pmatrix} = \sum_{\sigma \in S_k} \operatorname{sgn}(\sigma) \delta^{i_1}_{j_{\sigma(1)}} \delta^{i_2}_{j_{\sigma(2)}}  \cdots \delta^{i_k}_{j_{\sigma(k)}} . \]
In particular, when $2q > m$, $\delta^{i i_1 \cdots i_{2q}}_{j j_1 \cdots i_{2q}} = 0$.
Note that in dimension $4$, the solution space is generated by $g$ and its Ricci curvature $\Ric_g$. In higher dimensions, the expressions in Equation \eqref{eqn: generalized Ricci curvature} generalize the notion of Ricci curvatures. We denote  
\begin{equation*}
    \Big(\Ricqq_g\Big)_i^j = \delta^{i i_1 \cdots i_{2q}}_{j j_1 \cdots i_{2q}} R_{i_1 i_2}^{j_1 j_2} \cdots R_{i_{2q-1} i_{2q}}^{j_{2q-1} j_{2q}}
\end{equation*}
Using the contraction map, we can rewrite the above expression using double forms:
\[\Ric_g^\qq = \ctr[h]^{2q-1}\big(\Rm_g^q\big), \quad 2q \leq n.\]
We call $\Ric_g^\qq$ the \bfemph{Ricci-$\boldsymbol{\qq}$ curvature}.
The trace of $\Ricqq_g$ is called the \bfemph{scalar-$\qq$ curvature} $\scal_g^\qq$, as it generalizes the scalar curvature. 

Analogous to the Einstein tensor, the \bfemph{Lovelock-$\boldsymbol{\qq}$ tensor} or generalized Einstein tensor is a symmetric, divergence-free tensor that depends on the metric $g$ and its derivatives up to order 2, given by
\[E_g^\qq = \Ric_g^\qq - \frac{1} {2q} \scal^\qq_g g.\]
Then a \bfemph{Lovelock metric} $g$ with coupling constants $\alpha_q$ and $\lambda$ is one that satisfies a \bfemph{Lovelock equation}
\begin{equation} \label{eqn: Lovelock equation}
    \sum \alpha_q E_g^\qq = \lambda g.
\end{equation} 

We refer to equations that only involve one Lovelock-$(2q)$ tensor as a \bfemph{pure Lovelock equation} and we refer to the equation $\Ricqq_g = \lambda^{(2q)} g$ for some constant $\lambda^{(2q)}$ as the \bfemph{Lovelock-$\boldsymbol{\qq}$ equation}. We will study a linear combination of Lovelock-$(2q)$ tensors, given by
\begin{equation} \label{eqn: Lovelock tensor}
    F_\alpha(g) = \sum \alpha_q \left( \Ricqq_g - \lambda^\qq g \right) -\frac{\alpha_q}{2q} \left( \scalqq_g - (n+1)\lambda^\qq  \right)g = 0,
\end{equation}
where 
\[\lambda^\qq = \left( -\frac{1}{2}\right)^q \frac{n! \; (2q)!}{(n-2q+1)!}\]
is chosen so that Equation \eqref{eqn: Lovelock tensor} is satisfied by the hyperbolic metrics for any choice of coefficients $\alpha = (\alpha_1, \alpha_2, \cdots, \alpha_q)$.
Comparing with Equation \eqref{eqn: Lovelock tensor}, we see that
\[\lambda(\alpha) = \sum_q \left( 1+\frac{n+1}{2q} \right) \alpha_q \, \lambda^\qq.\] 
\begin{ex}
    For a metric $g$ with constant sectional curvature $\kappa$, we have $\Rm_g = \frac{\kappa}{2}g^2$, and $\Ric^\qq_g = \kappa^q \lambda^\qq g$. Such a  metric $g$ satisfies the Lovelock equations for arbitrary choices of the coupling constants $\alpha_q$ if we replace $\lambda^\qq$ with $\left( -\frac{\kappa}{2}\right)^q \frac{n! \; (2q)!}{(n-2q+1)!}$.
\end{ex}

\begin{ex}
    In \cite{DH06, DH08}, the authors study Lovelock metrics of Taub-NUT type on 8-dimensional spaces of the form $X = \R_+ \times P$, where $P$ is a $U(1)$-bundle over some 6-dimensional Einstein manifold $B$. The Taub-NUT metrics are of the form
    \[g = F(r) (d\tau+N\cA)^2 + \frac{1}{F(r)} dr^2 + (r^2 - N^2) \check{g}_B.\]
    Here, $\check{g}_B$ is an Einstein metric on $B$, $\tau$ is the coordinate on the $U(1)$-fibers representing the proper time, $\cA$ denotes the connection 1-form on $P$, $N$ is a constant representing the NUT charge, and $F(r)$ is a radial function. 
    
    Explicit examples with base manifolds $B$ given by $\C P^3$, products of copies of $\C P^2, S^2$ and the torus $T^2$ have been worked out in \cite[\S 3]{DH08}.
\end{ex}

A conformally compact metric $g$ that satisfies a Lovelock equation is called a \bfemph{conformally compact Lovelock metric} or a \bfemph{Poincaré–Lovelock metric}. The corresponding Riemannian manifold $(X,g)$ with boundary is referred to as a \bfemph{conformally compact Lovelock manifold}.

\begin{ex}
    A model example of a conformally compact Lovelock metric with arbitrary choice of $\alpha$ (since it has constant sectional curvature) is the real hyperbolic space, realized as the unit ball $B^{n+1} \subset \mathrm{R}^{n+1}$ with the Poincar\'e metric $g_+ = \frac{4 g_{\mathrm{euc}}}{(1-|x|^2)^2}$ and with boundary~$S^n$. Here, $g_{\mathrm{euc}} = dx_0^2 + \cdots + dx_n^2 $ is the flat metric and $|x|^2 = x_0^2 + \cdots + x_n^2$ is the Euclidean norm. There are many choices of the boundary defining function, for example, $\rho_1 = 1-|x|^2$ and $\rho_2 = \frac{1-|x|}{1+|x|}$. The two corresponding boundary metrics belong to the same conformal class $[g_{S^n}]$ that contains the standard round metric on the sphere. When the conformal infinity $[h]$ is close enough (in the $C^{2,\gamma}$ sense) to the round metric, there is a unique (up to diffeomorphism) conformally compact Lovelock metric on the ball whose conformal infinity is the given $[h]$ \cite[Theorem 3.1]{Alb20}.
\end{ex}

Hereafter we assume that $X$ is a manifold with boundary and that we are given a fixed boundary defining function $x$ of $X$. A notion closely related to a conformally compact Lovelock metric is an asymptotically hyperbolic metric. We say that $g$ is an \bfemph{asymptotically hyperbolic metric} if its sectional curvature approaches the negative constant number $-\kappa$ as $x \to 0$. We assume $\kappa = 1$ for simplicity, as this can always be achieved by rescaling the metric.

Consider a conformally compact metric $g = x^{-2} \gbar$. The asymptotic behavior of its Riemannian curvature as $x \to 0$ is given by $\Rm = -\frac{\kappa}{2} \gbar^2 x^{-4} + O(x^{-3}), \kappa > 0$ \cite{Maz88}, where~$\kappa$ is the conformal invariant $|dx|^2_{\gbar} \big|_{x=0} = |dx/x|^2_{g} \big|_{x=0} = \gbar^{ij} x_i x_j \big|_{x=0}$. The geometric meaning of the negative norm $-|dx|^2_{\gbar} \big|_{x=0}$ is the asymptotic sectional curvature of $g$. 
If the Lovelock equation $F_\alpha(g) = 0$ holds, then the asymptotic behavior of $\Rm_g$ implies 
\begin{equation} \label{eqn: asymptotic behavior of F_g}
    F_\alpha(g) = \sum \lambda^\qq \alpha_q \left( 1 - \frac{n+1}{2q} \right) (\kappa^q - 1) \gbar x^{-2} + O(x^{-1}) = 0.
\end{equation}
Hence $\sum_q \lambda^\qq \alpha_q \left( 1 - \frac{n+1}{2q} \right) (\kappa^q - 1) = 0$. In particular, a pure Lovelock equation implies~$\kappa = 1$, so a conformally compact pure Lovelock metric is asymptotically hyperbolic.

Following \cite{Alb20}, we define the \bfemph{limit set of sectional curvatures} for a given $q$-tuple~$\alpha$ as
\begin{multline*}
    \LimSec(\alpha) = \bigg\{\kappa > 0: \sum_q \lambda^\qq \alpha_q \left( 1 - \frac{n+1}{2q} \right) \big( \kappa^q-1 \big) = 0, \\
    \text{ and } A_1(\alpha,\kappa) = \sum \alpha_q \Big(-\frac{\kappa}{2} \Big)^{q-1} \frac{(n-2)!}{2} \frac{(2q)!}{(n-2q)!} \neq 0 \bigg\}.
\end{multline*}
We only work with $\alpha$ such that $\LimSec(\alpha) \neq \emptyset$ so there exists $\kappa \in \LimSec(\alpha)$. By Theorem~1 \cite{Alb20}, there exists a solution to the Lovelock equation 
\[\sum_q \alpha_q E^\qq_g = \lambda(\alpha),\] 
up to higher order terms of $x$, whose sectional curvature converges to $-\kappa$ as $x \to 0$. Again, by rescaling the metric, we can assume $\kappa = 1$ for simplicity. A particular case is when the constants $a_q$ are alternating, $1 \in \LimSec(\alpha)$.
\section{Gauge condition}
In this section, we examine the gauge condition necessary for applying elliptic regularity. We will see that, in general, the Lovelock tensor 
\[F_\alpha(g) = \sum \alpha_q \left( \Ricqq_g - \lambda^\qq g \right) - \frac{\alpha_q}{2q} \left( \scalqq_g - (n+1)\lambda^\qq  \right)g,\]
is not elliptic as operators of $g$. To overcome this issue, Albin \cite{Alb20} following \cite{GL91} applied DeTurck's trick to modify the Lovelock tensor. We start by introducing the following operators:

\begin{align*}
    &\delta_g: \text{symmetric 2-tensors} \to \text{one forms}, &&(\delta_g t)_i = - g^{jk}\,  t_{ij,k},\\
    &\delta_g: \text{one forms} \to \text{functions}, &&\delta_g\omega = - g^{jk} \, \omega_{j,k},\\
    &\delta_g^*: \text{one forms} \to \text{symmetric 2-tensors}, &&(\delta_g^* \omega)_{ij} = \frac{1}{2}(\omega_{i,j} + \omega_{j,i}),\\
    &gt^{-1}: \text{one forms} \to \text{one forms}, &&(gt^{-1}\omega)_i = g_{ij} \, t^{jk} \, \omega_k,\\
    &\cG_g: \text{symmetric 2-tensors} \to \text{symmetric 2-tensors}, &&(\cG_g \Phi)_{ij} = \Phi_{ij} - \frac{1}{2}g^{kl} \, \Phi_{kl} \, g_{ij}.
\end{align*}
We define the \bfemph{Bianchi operator} to be $B_g = \delta_g \cG_g$. By the contracted Bianchi identity, we have $B_g(g) = 0$.

When viewed as an operator on metrics, the linearization of the Lovelock tensor at an asymptotically hyperbolic metric $g_0$ is given in \cite[Lemma 3.2]{Alb20}. Let  $r$ be a symmetric 2-tensor of the form $r = x^N \bar{r}$ for some~$N$, where $x$ is a special boundary defining function for $g_0$ and $\bar{r} \in C^2(X;\sS^2 TX)$. Then the linearization of $F_\alpha(g)$ is given by
\begin{multline*}
    D\Big(F_\alpha(g)\Big)_{g_0}(r) \\
    = -\sum_q \frac{\lambda^\qq \alpha_q}{2n(n-1)} \bigg[ q(n-1)\left( 1 - \frac{n+1}{2q} \right) (\Delta_{g_0} + 2n)(\ctr[g_0](r) g_0) + (n-2q+1) (\Delta_{g_0} - 2)(r_0)  \bigg]\\
    + (c_1 \delta_{g_0}^* + c_2 g\delta_{g_0}) B_{g_0}(r) + O(x^{N+1}),
\end{multline*}
where the symmetric 2-tensor $r$ decomposes into a trace term $\ctr[g_0](r) g_0$ and a traceless term~$r_0 = r - \ctr[g_0](r) g_0$. Moreover, if $g_0$ has constant sectional curvature, then the $O(x^{N+1})$ term is identically zero.

Note that the linearization is almost a Laplace type operator except there are second order operators $\delta_{g_0}B_{g_0}(r)$, $\delta^*_{g_0}B_{g_0}(r)$ and $O(x^{N+1})$. Within a sufficiently small collar neighborhood of the boundary and under the assumption that we are looking at the linearization at asymptotically hyperbolic metrics, the term $O(x^{N+1})$ is small. In order to obtain a linearization that is dominated by the Laplace type operators, we use DeTurck's trick to get rid of the terms involving the Bianchi operator $B_g$. We define the \bfemph{modified Lovelock tensor}
\begin{equation} \label{eqn: modified Lovelock tensor}
    Q_\alpha(g,t) = F_\alpha(g) - \Phi_\alpha(g,t), 
\end{equation}
where $t$ is an auxiliary metric and
\begin{multline*}
    \Phi_\alpha(g,t) = (c_1 \delta_g^* + c_2 g\delta_g) \Big( gt^{-1} B_g(t) \Big)\\
    = -\sum_q \frac{\lambda^\qq}{n(n-1)} \alpha_q \bigg[ (n-2q+1) \delta^*_g - \Big( (q-1) - \frac{n-1}{2} \Big) g\delta_g \bigg] \Big( gt^{-1} B_g(t) \Big).
\end{multline*}
Notice that the term $\Phi_\alpha$ is chosen so that the linearization $D_1\Phi_\alpha$ with respect to the first argument cancels the second order terms involving $B_g(r)$ that obstruct ellipticity. More precisely, by Lemma 3.2 and 3.3 in \cite{Alb20}, the linearization of $Q_\alpha(g,t)$ with respect to the first argument is given by
\begin{align*}
    \Big( D_1Q_\alpha(g,t)\Big)(r) = \frac{A_1(\alpha)}{4} \bigg[ -(n-1) (\Delta_{g_0} + 2n)(\ctr[g_0](r) g_0)  + 2 (\Delta_{g_0} - 2)(r_0) \bigg] + O(x^{N+1}).
\end{align*}
Here the constant $A_1(\alpha)$ is the same as that defined in \cite[p.4]{Alb20} with $\kappa = 1$:
\begin{align*}
    A_1(\alpha) &= \sum_q \alpha_q \left( -\frac{1}{2} \right)^{q-1} \frac{(n-2)!}{2} \frac{(2q)!}{(n-2q)!}.
\end{align*}
So within a sufficiently small collar neighborhood of $M$, the linearization of the modified Lovelock tensor at an asymptotically hyperbolic Lovelock metric is elliptic. 
\section{Invertibility of Laplace operators} \label{sec: invertibility of Laplace operators}

\subsection{Indicial operators}
The elliptic regularity that we are going to use to demonstrate the regularity of asymptotically hyperbolic metrics hinges upon Laplace type operators of the form $\Delta + \cK$, where $\Delta$ is the connection Laplacian and $\cK$ is a bundle endomorphism on the bundle of symmetric $r$-tensors. We will use the Green's operator to construct an approximating sequence of the metric in \S\ref{sec: phg}. These Laplace type operators have been studied in \cite{GL91, Lee06}. We will often use the case when $\cK$ is a constant, either $2n$ or $-2$. 

Let $X$ be a manifold with boundary $\pd X = M$. Let 
\[P: C^\infty(X;E) \to C^\infty(X;F)\]
be a 0-differential operator of order $k$ acting between tensor bundles $E$ and $F$.
\begin{defn}
    The \bfemph{indicial operator} of $P$ at $\mu\in\C$ is given by
    \begin{align*}
        I_\mu(P): E|_{M} &\to F|_{M}\\
        u &\mapsto x^{-\mu} P(x^\mu \bar{u})|_M,
    \end{align*}
    where $u$ is a smooth section of $E|_M$ and $\bar{u}$ is a smooth section of $E \to X$ that extends $u$ so that $\bar{u}|_M = u$. An \bfemph{indicial root} $\mu$ is a complex number at which $I_\mu(P)$ is singular. 
\end{defn}

\begin{ex}
    For a $0$-differential operator $P$ of order $k$, which in local coordinates can be written as
    \[P = \sum_{j+|\gamma|\leq k} C_{j,\gamma}(x,y) (x\pd_x)^j (x\pd_y)^\gamma,\]
    the indicial operator of $P$ at $\mu$ is given by 
    \[I_\mu(P) = \sum_{j\leq k}  C_{j,0}(0,y) \mu^j.\]
\end{ex}    

Suppose $P$ is acting on sections of a tensor bundle $E$ of type $(r_1, r_2)$, where $r_1$ is the covariant rank and $r_2$ is the contravariant rank. We call $r = r_1 - r_2$ the \bfemph{weight} of $E$. If~$P$ is a formally self-adjoint operator, then by Corollary 4.5 in \cite{Lee06}, the indicial roots are symmetric with respect to the line $\left\{\operatorname{Re}(\mu) = \frac{n}{2} - r\right\}$. For concreteness, we denote the set of indicial roots labeled by some index set $I$ by $\big\{ \mu^{(i)}_+, \mu^{(i)}_- : i \in I \big\}$. We call the indicial roots $\mu^{(i)}_+$ with $\operatorname{Re}\big( {\mu^{(i)}_+} \big) \geq \frac{n}{2} - r$ the \bfemph{upper indicial roots} and the indicial roots $\mu^{(i)}_-$ with $\operatorname{Re}\big( {\mu^{(i)}_-} \big) \leq \frac{n}{2} - r$ the \bfemph{lower indicial roots}.

\begin{defn}        
    We define $\mu_+ = \inf \{\mu^{(i)}_+\}$ and the \bfemph{indicial radius} 
    \[R = R(P) = \mu_+ - \frac{n}{2} + r.\]
\end{defn}

We specialize in Laplace operators $\Delta_g + \cK$ acting on symmetric $r$-tensors on a connected asymptotic hyperbolic manifold $X^{n+1}$ of class $C^{l,\gamma}$, with $l \geq 2$ and $\gamma \in [0,1)$ . The following summarizes the results in \cite[Chapter 7]{Lee06} which help to compute the indicial roots of the operator $\Delta_g + c$, where $c$ is a constant.
\begin{prop} 
Let $c \in \mathbb{R}$. Consider the Laplace operator $\Delta_g + \cK$ acting on covariant symmetric $r$-tensors.  
    \begin{enumerate}
        \item $I_\mu(\Delta_g + \cK) = I_0(\Delta_g + \cK) + \mu(n-\mu-2r)$;

        \item on trace free symmetric $r$-tensors, $R(\Delta_g) = \sqrt{\frac{n^2}{4} + r}$.

        \item let $R = R(\Delta_g + \cK) > 0$, then the indicial radius $R' = R(\Delta_g + \cK+c)$ is positive if and only if $c+R^2 > 0$. In this case $R' = \sqrt{c+R^2}$.
    \end{enumerate}
\end{prop}

We state the indicial roots of $\Delta_g + c$ computed in \cite[pp.199--202]{GL91}. 
\begin{prop}[{\cite[pp.199--202]{GL91}}]
    Let $P = \Delta_g + c$, $c\in \R$. If $r_1 = r_2 = 0$, that is, $P$ acts on the set of smooth functions $C^\infty(X)$, the indicial roots of $P$ are 
    \[\xi_\pm = \frac{1}{2}(n \pm \sqrt{n^2 + 4 c}).\]
    On the bundle of symmetric 2-tensors $\sS^2$, we consider the decomposition into trace and trace-free terms $\sS^2 = \sG \oplus \sS_0^2$. The bundle $\sS_0^2|_M$ decomposes further into 
    \[\sS_0^2|_M = \mathscr{V}_1 \oplus \mathscr{V}_2 \oplus \mathscr{V}_3,\] where
    \begin{align*}
        \mathscr{V}_1 &= \{\bar{q}_{ij} \mid \gbar^{ij} \pd_i x \,\bar{q}_{ij} = 0, \tr_{\gbar}\bar{q} = 0 \},\\
        \mathscr{V}_2 &= \{\bar{q}_{ij} \mid \bar{q}_{ij} = \lambda\big( (n+1) \pd_i x \, \pd_j x - \gbar_{ij} \big), \lambda \in \mathbb{R} \},\\
        \mathscr{V}_3 &= \{\bar{q}_{ij} \mid \bar{q}_{ij} = v_i \pd_j x + v_j \pd_i x, v \in T^*X|_M, \gbar^{ij} v_i \pd_j x = 0 \},
    \end{align*}
    and $x$ is a boundary defining function. For $1 \leq i \leq 3$, let $V_i$ denote the subspace of~$C^2(X,\sS_0^2)$ consisting of those tensors which belong to $\sV_i$ at $M$. The indicial roots $\mu^{(i)}_\pm$ corresponding to solving $I_{\mu^{(i)}}(P|_{V_i}) = 0$ are given as follows:

    \begin{alignat*}{2}
        & \mu^{(0)}_\pm  = \frac{n-4 \pm \sqrt{n^2 + 4c}}{2}, &\qquad& \mu^{(1)}_\pm = \frac{n-4 \pm \sqrt{n^2 + 4c + 8}}{2}, \\
        & \mu^{(2)}_\pm  = \frac{n-4 \pm \sqrt{n^2 + 8n + 4c + 8}}{2}, &&  \mu^{(3)}_\pm  = \frac{n-4 \pm \sqrt{n^2 + 4n + 4c + 12}}{2}.
    \end{alignat*}
\end{prop}

\begin{ex} \label{ex: laplacian is an isom}
    We will be using the elliptic regularity of the following Laplace type operator acting on a symmetric 2-tensor $r = \tr_g(r) g + r_0$,
    \[P(r) = C_1 (\Delta_g + 2n)(\tr_g(r) g)  + C_2(\Delta_g - 2)(r_0),\]
    where $C_1, C_2$ are constants. The indicial roots of $P$ are given by
    \begin{alignat*}{2}
        & \mu^{(0)}_\pm = \frac{1}{2} (n - 4 \pm \sqrt{n^2 + 8n}), &\qquad& \mu^{(1)}_\pm = n-2, -2,\\
        & \mu^{(2)}_\pm = \frac{1}{2} (n - 4\pm \sqrt{n^2 + 8n}), && \mu^{(3)}_\pm = n-1, -3,
    \end{alignat*}
    \begin{center}
        \includegraphics[width=0.8\textwidth]{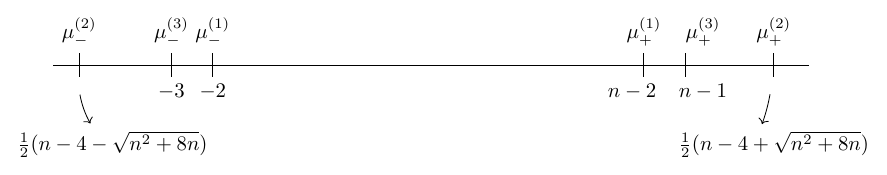}
    \end{center}

Suppose $X$ is an asymptotically hyperbolic manifold of class $C^{\ell, \zeta}$, with $\ell \geq 2$. We denote by $P|_{\sG}$ and $P|_{\sS_0^2}$ the restrictions of $P$ to the subbundles of $\sS^2$. Then by Proposition D in \cite{Lee06}, the restriction $P|_{\sG}(r) = C_1(\Delta_g + 2n)(\tr_g(r)g)$ satisfies the hypothesis of Theorem~C in \cite{Lee06}. Hence the natural extension 
\[P|_{\sG}(r): C^{k,\xi}_\delta(X; \sG) \to C^{k-2,\xi}_\delta(X; \sG)\]
is Fredholm when $k \geq 2, 0 < \delta < n$. Similarly, Proposition E and Theorem C in \cite{Lee06} imply the natural extension 
\[P|_{\sS_0^2}: C^{k,\xi}_\delta(X; \sS_0^2)  \to C^{k-2,\xi}_\delta(X; \sS_0^2) \]
is Fredholm when $k \geq 2, 0 < \delta < n$. Since the operator $P$ respects the orthogonal decomposition, 
\[P: C^{k,\xi}_\delta(X; \sS^2) \to C^{k-2,\xi}_\delta(X; \sS^2)\]
is Fredholm on symmetric 2-tensors, when $k \geq 2, 0 < \delta < n$.

Moreover, in the case where $X = B_+^{n+1}$, as defined in \S\ref{sec: conformal geometry}, let ${\prescript{0}{}{C}}^{k,\xi}_{\delta_1, \delta_2}$ denote the space of H\"older functions/sections that satisfy the Dirichlet boundary condition on the interior boundary $\pd_{\mathrm{int}} B_+^{n+1} = \{x_0^2 + x_1^2 + \cdots + x_n^2 = 1\}$. When $\delta_1 \in (0,n), \delta_2 \in [0, n-1]$ and $k \geq 2$, Proposition 3.4 in \cite{BH14} implies 
\[\Delta_g + c: {\prescript{0}{}{C}}^{k,\xi}_{\delta_1, \delta_2} (B^{n+1}_+) \to C^{k-2,\xi}_{\delta_1, \delta_2} (B^{n+1}_+)\]
is an isomorphism if $\delta_1$ also satisfies $\frac{n-\sqrt{n^2+4c}}{2} < \delta_1 < \frac{n+\sqrt{n^2+4c}}{2}$. As we have seen in Example~\ref{ex: laplacian is an isom}, the indicial roots get shifted by 2 when the operator is applied to symmetric 2-tensors. Therefore, 
\[\Delta_g + 2n: {\prescript{0}{}{C}}^{k,\xi}_{\delta_1, \delta_2} (B^{n+1}_+,\sG) \to C^{k-2,\xi}_{\delta_1, \delta_2}(B^{n+1}_+,\sG)\]
is an isomorphism when $\delta_1 \in \left(0,\frac{n-4+\sqrt{n^2+8n}}{2}\right), \delta_2 \in [0,n-1]$, and that 
\[\Delta_g-2: {\prescript{0}{}{C}}^{k,\xi}_{\delta_1, \delta_2} (B^{n+1}_+,\sS_0^2) \to C^{k-2,\xi}_{\delta_1, \delta_2}(B^{n+1}_+,\sS_0^2)\]
is an isomorphism when $\delta_1 \in \left(0,\frac{n-4+\sqrt{n^2-8}}{2}\right), \delta_2 \in [0,n-1]$. We conclude that within a sufficiently small collar neighborhood,
\[L_g: {\prescript{0}{}{C}}^{k,\xi}_{\delta_1, \delta_2} (B^{n+1}_+,\sS^2) \to C^{k-2,\xi}_{\delta_1, \delta_2}(B^{n+1}_+,\sS^2)\]
is an isomorphism when $\delta_1 \in \left(0,\frac{n-4+\sqrt{n^2-8}}{2}\right), \delta_2 \in [0,n-1]$.
\end{ex}

\subsection{Green's integral operators} \label{sec: Green's operator} 
Away from the indicial roots, the indicial operator of a Laplace-type operator is invertible on appropriate functional spaces, which allows us to construct Green's integral operators. In this section, we reformulate the results from \cite[\S6]{BH14} in terms $0$-differential operators.

Consider a manifold with boundary $X$ equipped with an asymptotically hyperbolic $g$. Suppose $E$ is a tensor bundle over $X$. Let $h$ denote the smooth metric on the boundary $M$ and let $g_0 = \frac{dx^2+h}{x^2}$. 
\begin{rmk}
    The notations used in the paper by Biquard and Herzlich \cite{BH14} can be translated as follows:
    \[x = e^{-s}, \quad x \pd_x = -\pd_s \quad \text{and} \quad \frac{dx^2 + h}{x^2} = ds^2 + e^{2s} h.\]
    Biquard and Herzlich set the dimension of $X$ to be $n$ rather than $n+1$, which results in a shift in the expression for the Laplacian and its indicial roots. 
\end{rmk}
We will consider Laplace type operators $L_{g_0}$ of hyperbolic metric $g_0$ acting on the symmetric $r$-tensors. Recall that within a collar neighborhood $M_{\leq x_0} = [0,x_0]\times M$ of the infinity boundary, the metric $g_0$ is of product type $\frac{dx^2+h}{x^2}$. 

We can trivialize a tensor bundle $E$ on $M_{\leq x_0}$ along each ray $\R_+ \times \{y\}$, where $y$ is a choice of a parallel basis for the connection $\nabla^g$. In a trivializing chart $\cU$, the derivative $\nablatilde_x$ becomes the ordinary derivative $\pd_x$ and the Laplace type operator $L_{g_0}$ is asymptotic to an operator of the form
\[-(x \pd_x)^2 + n x \pd_x + \widetilde{A} + \cK_0 + L^{\trans},\]
where $\widetilde{A}$ are the constant zeroth-order terms in $\Delta_g$ (they are the same as those of the hyperbolic space operator), $\cK_0$ contains the constant curvature terms (these are the curvature terms of hyperbolic space), and $L^{\trans}$ denotes a linear combination of the 0-differential operators $x\pd_y, (x\pd_y)^2$.

As $x \to 0$, the behavior of $L_g$ on $\cU$ is dominated by the operator
\[\sI(L_g) = -(x\pd_x)^2 + nx\pd_x+ \widetilde{A} + \cK_0.\]

The right inverses of the Laplace-type operators $-\sI(L_g)$ for a pair of real numbers~$(\alpha_-, \alpha_+)$ with $\alpha_- < \alpha_+$, given in \cite{BH14}, are constructed from the Green’s integral operators:
\begin{align} 
    G_{\infty}(u)(x) &= \frac{1}{\alpha_+ - \alpha_-} \left( x^{\alpha_-} \int_0^x -\widetilde{x}^{-\alpha_- -1} \, u(\widetilde{x}) \,d\widetilde{x} -  x^{\alpha_+} \int_0^x \widetilde{x}^{-\alpha_+-1} \, u(\widetilde{x}) \,d\widetilde{x} \right),  \label{eqn: Green's operator, G_infty}\\
    G_{0}(u)(x) &= \frac{1}{\alpha_+ - \alpha_-} \left( x^{\alpha_-} \int_0^x -\widetilde{x}^{-\alpha_- -1} \, u(\widetilde{x}) \,d\widetilde{x} -  x^{\alpha_+} \int_{x'}^x \widetilde{x}^{-\alpha_+-1} \, u(\widetilde{x}) \,d\widetilde{x} \right),  \label{eqn: Green's operator, G_0}
\end{align}
where $x'$ is a small positive real number. These integral operators have the following properties:
\begin{prop}[Green's operator on the semi-ball {\cite[p.833]{BH14}}] \label{prop: choice of G} 
    Let $G_\infty$ and $G_0$ be the integral operators defined by Equations \eqref{eqn: Green's operator, G_infty} and  \eqref{eqn: Green's operator, G_0}, respectively, on the semi-ball.
    \begin{enumerate} 
    \item If $\xi > \alpha_+$, the operator $G_\infty$ maps the weighted spaces $C^\infty_{\xi, \zeta}, 
    \Adot_{\xi, \zeta}$ and $\Adot[\xi]_\zeta$ 
    into themselves;

    \item If $\alpha_- < \xi \leq \alpha_+$, the operator $G_0$ maps $\Adot_{\xi, \zeta}
    $ into itself and maps the spaces $C^\infty_{\xi, \zeta}$ , $\Adot[\xi]_{\zeta}$ 
    into themselves only if $\xi < \alpha_+$; at the critical value $\xi = \alpha_+$, a logarithmic term $x^{\alpha_+} \log x$ appears in $G_0 (x^{\alpha_+})(x)$, so $G_0$ does not preserve the function spaces.
\end{enumerate}

The operators $G_\infty$ and $G_0$ serve as right inverses to the model operator $(x\pd_x)^2-nx\pd_x-c$ on the specified spaces mentioned above, assuming that the constant $c$ and the pair $(\alpha_-, \alpha_+)$ satisfy the condition 
\[\alpha_\pm = \frac{n \pm \sqrt{n+4c}}{2}.\] 
\end{prop}

Let $\lambda$ denote an eigenvalue of $\widetilde{A} + \cK_0$ acting on symmetric $r$-tensors. The associated model operator
\[L_\lambda = (x\pd_x)^2-nx\pd_x-\lambda\]
has indicial roots $\alpha_\pm = \frac{n\pm\sqrt{n^2+4\lambda}}{2}$, and the corresponding integral operators $G_\infty$ and $G_0$ serve as right inverses to $L_\lambda$ on the H\"older and polyhomogeneous spaces identified in Proposition~\ref{prop: choice of G}.

\begin{prop}[Extension of Green’s operators to tensor bundles {\cite[p.833]{BH14}}] \label{prop: extension of Green's operator to tensor bundles}
    The construction of the integral operators $G_\infty$ and $G_0$ extends to sections of a geometric tensor bundle $E$ over $M_{\leq x_0}$ by radial trivialization. Over each ray $\R_+ \times \{y\}$, $\widetilde{A}+\cK_0$ is constant. Let $E = \bigoplus_\lambda E_\lambda$ be the decomposition into eigenspaces of $\widetilde{A}+\cK_0$ and define the Green's integral operator $\cG$ as a direct sum of the integral operators $G_\infty$ and $G_0$ on each $E_\lambda$. Then, on each appropriate function space as in Proposition \ref{prop: choice of G}, the operator $\cG$ satisfies 
    \[-\sI(L_{g_0})\cG = \id.\]
\end{prop}

We state results on Laplace operators from \cite{BH14} that will be useful in showing the regularity of Lovelock metrics. 
Let $B_+^{n+1}$ be a semi-ball with asymptotically hyperbolic metric $\phi$. Consider a Laplace type operator 
\begin{equation} \label{eqn: Laplace type operator}
    L = \Delta_\phi + \cK = -(x\pd_{x})^2 + nx\pd_x + \widetilde{A} + \cK_0 + L^{\trans}
\end{equation}
acting on symmetric $r$-tensors.
\begin{lem}[{\cite[Lemma 3.5, Lemma 7.1]{BH14}}] \label{lem: BH14, Laplacian is isom}
    Suppose $L$ is invertible on $L^2$. Let $\lambda$ be the smallest eigenvalue of $\widetilde{A} + \cK$ on symmetric $r$-tensors. Suppose the double weight $(\delta_1, \delta_2)$ satisfies $0 < \delta_1 < n$, and $ 0 \leq \delta_2 \leq n-1$. Set $\mu_\pm = \frac{n}{2} \pm \sqrt{\frac{n^2}{4} + \lambda}$.

    \begin{enumerate}
        \item If $k \geq 2, \alpha \in (0,1), \lambda \geq 0$ and $\mu_- < \delta_1 < \mu_+$,
        then $L: C^{k,\alpha}_{\delta_1, \delta_2}(B_+^{n+1}) \to C^{k-2,\alpha}_{\delta_1, \delta_2}(B_+^{n+1})$ is an isomorphism.
        \item If $2 \leq k \leq l$ and $\alpha \in (0,1)$. Suppose $u$ is a section satisfying 
        \[u \in C^{k,\alpha}_{\delta_1', \delta_2}(B_+^{n+1}) \text{ and } L(u) \in C^{l-2,\alpha}_{\delta_1, \delta_2}(B_+^{n+1}) \text{ with } \mu_- \leq 0 \leq \delta_1' \leq \delta_1 < \mu_+.\]
        Then $u \in  C^{l,\alpha}_{\delta_1, \delta_2}(B_+^{n+1}) $.
    \end{enumerate}
\end{lem}

We conclude this section with a regularity result involving the operator algebra $\cQ$ defined \S\ref{sec: operator algebra Q}, which will be used to control the decay of transverse derivatives in the proof of the main theorem.

\begin{prop}[{{\cite[Proposition 7.2]{BH14}}}] \label{prop: BH14, decreasing in transverse derivatives}
    Suppose the double weight $(\delta_1, \delta_2)$ satisfies the assumption in Lemma \ref{lem: BH14, Laplacian is isom} (i). Suppose $k\geq 2$ and $\alpha \in (0,1)$. Let $L$ be the operator given by Equation \eqref{eqn: Laplace type operator}. Suppose the smallest upper critical weight $\mu_+$ of $L$ is greater than or equal to $1$. If~$u \in C^{k, \alpha}_{\delta_1, \delta_2}(B^{n+1}_+)$ satisfies 
    \[L(u) = \ell(u) + q(u) + f, \]
    where $\ell$ is a linear differential operator with polyhomogeneous coefficients which belongs to~$x \cQ_1$; the operator $q$ is of order at least two and has polyhomogeneous coefficients; and $f$ belongs to $\cA_{\mu_+ + \eta, \delta_2} (B^{n+1}_+)$ for some $\eta > 0$, then, for all $\delta < 1$ and all operators $Q \in \cQ$, 
    \[Q (\nablatilde_{\pd_y})^k u \in C^0_{\mu_+ + \delta, \delta_2}(B^{n+1}_+), \quad \forall k \in \mathbb{N}.\]
\end{prop}
\section{Establishing the polyhomogeneous expansion} \label{sec: phg}
In this section, we prove a regularity result for asymptotically hyperbolic Lovelock metrics by adapting the method developed by Biquard and Herzlich in the context of asymptotically hyperbolic Einstein metrics \cite{BH14}. Specifically, we aim to demonstrate that such Lovelock metrics admit a polyhomogeneous expansion in a neighborhood of the conformal boundary.
\begin{thm}[Regularity of asymptotically hyperbolic Lovelock metrics]
    Let $X$ be a manifold with boundary equipped with an asymptotically hyperbolic Lovelock metric $g$, which induces a smooth conformal infinity $[h]$. Then $g$ is polyhomogeneous in a neighborhood of the boundary.
\end{thm}
The polyhomogeneity of such metrics is a local problem, so we only need to verify this property within a neighborhood of an arbitrary boundary point. (See Theorem \ref{main thm: phg of semi-ball} below.) We first set up notations for the local regularity problem. Let $Y$ be a manifold diffeomorphic to the semi-ball 
\[B^{n+1}_+ = \{x_0^2 + x_1^2 + \cdots + x_n^2 < 1, \, x_0 > 0\}.\]
Its boundary at infinity and the interior boundary of $Y$ are described as follows:  
\begin{align*}  
    \pd_\infty Y &\cong \pd_\infty B^{n+1}_+ = \{x_1^2 + \cdots + x_n^2 < 1, \, x_0 = 0\}, \\  
    \pd_{\mathrm{int}} Y &\cong \pd_{\mathrm{int}} B^{n+1}_+ = \{x_1^2 + \cdots + x_n^2 = 1, \, x_0 > 0\}.  
\end{align*}  
For convenience, we denote the boundary defining function by $x = x_0$. The boundary metric on $Y$ restricts to the boundary metric on $\pd_\infty Y$, which we will denote by $h$. The simplest asymptotically hyperbolic metric on $Y$ consistent with $h$ is 
\[\phi_0 = \frac{dx^2 + h}{x^2}.\]
By appropriately modifying the coordinates if necessary, we can assume that at the origin, the metric $h$ simplifies to  
\[h = dx_1^2 + \cdots + dx_n^2.\]
Next, consider the dilation map $f_\epsilon$, defined for $\epsilon > 0$ as  
\[f_\epsilon(x_0, \ldots, x_n) = (\epsilon x_0, \ldots, \epsilon x_n).\]
Under this transformation, the rescaled metric satisfies 
\[\lim_{\epsilon \to 0} f_\epsilon^* \phi_0 = \phi_\hyp = \frac{dx_0^2 + \cdots + dx_n^2}{x_0^2}.\]

Consider an asymptotically hyperbolic Lovelock metric $\phi$ on $Y$, in fact it is enough to assume that 
\begin{equation} \label{eqn: asymptotically hyperbolic Lovelock}
    \phi - \phi_0 \in C^{2,\gamma}_\delta(X; \sS^2 T^*X), \text{ for some }  \delta >0, \text{ and } \gamma \in (0,1), \quad \text{and} \quad F_\alpha(\phi) = 0.
\end{equation} 

We first show that there is a local gauge for the action of diffeomorphisms on the metric~$f_\epsilon^* \phi$ for small enough $\epsilon$ and $\phi$ a Lovelock metric. Let 
\[B_{\phi} = \delta_{\phi} + \frac{1}{2} d \tr_{\phi} \]
be the Bianchi operator acting on symmetric two-tensors. 

\begin{lem}[Local gauge] \label{lem: choice of gauge} 
    Let $\phi$ be an asymptotically hyperbolic Lovelock metric on $Y$. Fix some auxiliary metric $t$, such that $t - \phi_0 \in C_{\delta, \delta}^{\infty}$. For $\epsilon>0$ small enough, there exists a diffeomorphism $\Psi$ that equals the identity map when restricted to the boundary at infinity~$\pd_\infty Y$ and satisfies 
    \[B_{\Psi^* f_\epsilon^* \phi} (f_\epsilon^* t) = 0.\]
\end{lem}
\begin{proof}
    We assume the dilation $f_\epsilon$ is applied for sufficiently small $\epsilon$, and we will omit it in the subsequent parts of the proof. We follow the approach outlined in \cite[\S5]{And10} and \cite[Proposition A.2]{Ozu22} for Einstein metrics, and extend this argument to Lovelock metrics. The only adjustment required is to verify that the linearization of $Q_\alpha(\cdot, t)$ with respect to the first argument is invertible at the metric $\phi$. Indeed, we see from \cite[Lemma 3.4]{Alb20}, the linearization is given by:
    \begin{align*}
        L_{\phi}(r) &= \Big( D_1Q_\alpha(\cdot,t)\Big)_{(\phi,\phi)}(r) \\
        &= \frac{A_1(\alpha)}{4} \bigg[ -(n-1) (\Delta_{\phi} + 2n)(\ctr[\phi](r) \phi)  + 2 (\Delta_{\phi} - 2)(r_0) \bigg] + O(x^{N+1}).
    \end{align*}
    As discussed in Example \ref{ex: laplacian is an isom}, the Laplace term of $L_{\phi_0}$ is an isomorphism as a map 
    \[{\prescript{0}{}{C}}^{k,\xi}_{\zeta, \zeta} (Y,\sS^2) \to C^{k-2,\xi}_{\zeta, \zeta}(Y,\sS^2), \text{ for } k \leq 2, \xi \in (0,1).\]
    This, together with the assumption that $\epsilon$ is sufficiently small, guarantees the invertibility of $L_{\phi}$. Then the Ebin slice theorem \cite[Theorem 7.1]{Ebi70} provides a diffeomorphism $\Psi$ such that $B_{\Psi^*\phi}(t) = 0$.
\end{proof}

\begin{rmk} \label{rmk: gauge condition}
    To simplify the notation, we will assume that $\phi$ has been replaced by $\Psi^* f_\epsilon^*\phi$ and $t$ has been replaced by $f_\epsilon^*t$ so that 
    \begin{equation} \label{eqn: gauge condition}
        B_\phi(t) = 0
    \end{equation}
    for some auxiliary metric $t$, which we will specify in Theorem \ref{thm: first approximation}. We will refer to Equation~\eqref{eqn: gauge condition} as the \bfemph{gauge condition}.
\end{rmk}

\subsection{Nonlinear elliptic equations} 
Before stating and proving the next lemma, we recall the definition of a nonlinear elliptic system and state a regularity result for solutions of nonlinear elliptic equations. We refer the reader to \cite[\S2]{DeT81} and \cite[Appendix J]{Bes87}.

Given a system $\cF = (\cF_1, \cdots, \cF_p)$ of partial differential equations 
\begin{equation} \label{eqn: nonlinear pde}
    \cF_i(x,D^\xi u) = 0, \quad i = 1, \cdots p
\end{equation}
of order $k$, we seek a solution $u = \big(u^1(x), \cdots, u^q(x)\big)$.
We assume that the $\cF_i$ are $C^{\ell, \epsilon}$ in the~$x$ variable for some $\ell, \epsilon>0$ and $C^\infty$ in $u$ and $D^\alpha u$. We say the system \eqref{eqn: nonlinear pde} is \bfemph{elliptic} at the point $x_0$ for the solution $u_0$ if the linearizations
\[L_i (r) = \sum_{|\zeta| = k, j \leq q} \frac{\pd \cF}{\pd(D^\zeta u^j)}(x_0, D^\xi u_0) D^\zeta r, \quad i = 1, \cdots, p\]
are elliptic. The \bfemph{principal symbol} of \eqref{eqn: nonlinear pde} is the principal symbol of its linearization. 
Here is a regularity result for a nonlinear elliptic system that we will be using to prove Lovelock metrics are smooth in the gauge given by DeTurck's trick.
\begin{thm}[{\cite[Theorem 41, p.467]{Bes87}}] \label{thm: Besse}
    Assume that the system $\cF(x,s)$ is $C^\infty$ in all of its variables for $x$ in an open set $\Omega \subset \R^n$ and all $s$, and that $u = (u^1, \cdots, u^q) \in C^k(\Omega)$ is an elliptic solution of the $k$-th order equation $\cF(x, \pd^k u) = 0$. Then $u \in C^\infty(\Omega)$. 
\end{thm}

\subsection{Regularity results}
The first local regularity result is that a Lovelock metric $\phi$ which is close to $\phi_0$ is smooth in the gauge given by DeTurck's trick. One can compare this result with Lemma 5.4 in \cite{BH14}.

More precisely, we have the following lemma. 
\begin{lem} \label{lem: AHE implies smooth difference}
    Suppose $\phi$ is an asymptotically hyperbolic Lovelock metric of class $C^2$.  Let $t$ be an auxiliary metric satisfying the gauge condition and $C^\infty_\zeta$ close to $\phi_0$ for $\zeta > 0$. Then 
    \[\phi-t \in C^\infty_\zeta(Y; \sS^2 T^*Y).\]
\end{lem}

\begin{proof}
    The proof is similar to the regularity of Einstein metrics in modified harmonic gauge (\cite[Lemma 5.4]{BH14}). We apply regularity for solutions of nonlinear elliptic equations to the second-order equation
    \begin{equation} \label{eqn: proof Lovelock regularity g-t is smooth}
        Q_\alpha(\phi,t) = 0. 
    \end{equation}

    To determine whether $\phi$ is an elliptic solution of \eqref{eqn: proof Lovelock regularity g-t is smooth}, we consider its linearization, which is given by the linearization of $Q_\alpha(\phi,t)$ at an asymptotically hyperbolic metric $\phi$, given by
    \begin{align*}
        L_{\phi}(r)
        = \frac{A_1(\alpha)}{4} \bigg[ -(n-1) (\Delta_{\phi} + 2n)(\ctr[\phi](r) \phi)  + 2 (\Delta_{\phi} - 2)(r_0) \bigg] + O(x^{N+1}).
    \end{align*} 
    A similar argument to that in the proof of Lemma \ref{lem: choice of gauge} shows that this is an elliptic operator on $Y$. Hence, $\phi$ is an elliptic solution of \eqref{eqn: proof Lovelock regularity g-t is smooth}.

    Since Equation \eqref{eqn: proof Lovelock regularity g-t is smooth} is a polynomial in terms of $\phi,t, D^\alpha \phi, D^\beta t$, and the solution $\phi$ is $C^2$, Theorem~\ref{thm: Besse} says that $\phi \in C^\infty$. On the other hand, we have $\phi-\phi_0 \in C^{2,\gamma}_\zeta$ and $\phi_0 \in C^\infty$ as~$h$ is a smooth metric. This implies $\phi-\phi_0 \in C^\infty_\zeta$. Then 
    \[\phi - t = (\phi - \phi_0) + (\phi_0 - t) \in C^\infty_\zeta.\]
\end{proof}

\begin{thm} \label{thm: first approximation} 
    Let $Y$ be a manifold with boundary diffeomorphic to $B^{n+1}_+$. Let $x = x_0$ be the boundary defining function, $h$ be a smooth boundary metric on $\pd_\infty Y$ and $\phi_0 = \frac{dx^2+h}{x^2}$. Suppose $\phi$ is an asymptotically hyperbolic metric on $Y$ satisfying
    \[\phi - \phi_0 \in C^{2,\gamma}_\zeta(Y; \sS^2 T^*Y) \text{ for some } \zeta > 0, \gamma \in (0,1)  \quad \text{and} \quad F_\alpha(\phi) = 0,\]
    and suppose $\LimSec(\alpha) \neq \emptyset$.  
    Then there exists $\phibar$ satisfying the following: 
    \begin{enumerate}
        \item \textit{(polyhomogeneity)}. $F_\alpha(\phibar) \in \cA_{\nu, \zeta'}$ for some $\nu > \mu_+ = n-2$ and $\zeta' = \min\{\nu,\zeta\}$;
        \item \textit{(gauge)} $B_\phi(\phibar) = 0$. 
    \end{enumerate}
\end{thm}

\begin{proof}
We first construct $\phibar$ by solving for an approximate solution of the Lovelock equation: 
\begin{equation} \label{eqn: approximate Lovelock} 
    F_\alpha(\phi) = 0.
\end{equation}
Since $\LimSec(\alpha) \neq \emptyset$, by rescaling the metric, we can assume $\kappa = 1 \in \LimSec(\alpha)$. We can apply the argument in \cite[\S 2.2]{Alb20} and conclude that Equation \eqref{eqn: approximate Lovelock} yields
\begin{alignat}{2}
    \frac{1}{2}\Big( (1-k)& A_1(\alpha) + && 2(n-k) B_{1,2}(\alpha)\Big) \ctr[h] \big( h^{(k+1)} \big) \nonumber \\
    & && = \text{ terms involving lower order derivatives of } h + O(x).  \label{eqn: smooth solution of lovelock eqn}\\
    \frac
    {1}{2} (n-1-k)A_1(\alpha) \,&h^{(k+1)} +&&\frac{1}{2} \Big(A_1(\alpha) + 2(n-k) B_{1,2}(\alpha)\Big) h \ctr[h] \big( h^{(k+1)} \big) \nonumber \\
    &  &&= \text{ terms involving lower order derivatives of } h + O(x). \label{eqn: smooth solution of lovelock eqn 2} 
\end{alignat}
Here the constants $A_i(\alpha)$ and $B_{1,2}(\alpha)$ are the same as those defined in \cite[p.4, p.17]{Alb20} with $\kappa = 1$ and $\beta_q = -\frac{\alpha_q}{2q}$:
\begin{align*}
     A_1(\alpha) &= \sum_q \alpha_q \left( -\frac{1}{2} \right)^{q-1} \frac{(n-2)!}{2} \frac{(2q)!}{(n-2q)!},\\
     A_2(\alpha) &= \sum_q \alpha_q \left( -\frac{1}{2} \right)^{q-1} \frac{(n-2)!}{2} \frac{(2q)!}{(n-2q+1)!} (q-1), \\
     \lambda(\alpha) &= \sum_q \alpha_q \left( -\frac{1}{2} \right)^q \frac{n!\; (2q)!}{(n-2q+1)!},\\
     {B}_{1,2}(\alpha) &= A_2(\alpha) + \frac{1}{2n} \lambda(\alpha) = \sum_q \alpha_q \left( -\frac{1}{2} \right)^{q} \frac{(n-2)!}{2} \frac{(2q)!}{(n-2q)!}. 
\end{align*}
The assumption on $\LimSec(\alpha)$ implies that $A_1(\alpha) \neq 0$. Consequently, when $n-1-k \neq 0$, Equation \eqref{eqn: smooth solution of lovelock eqn} and the contraction of Equation \eqref{eqn: smooth solution of lovelock eqn} are linearly independent. This allows us to inductively solve for $\ctr[h] \big(h^{(k+1)}\big)|_{x=0}$ in terms of lower-order derivatives of $h$, and then use the second equation to determine $h^{(k+1)}|_{x=0}$. Thus, the boundary condition $h = x^2 \phi|_{x=0}$ determines an approximation of $\phi$ up to order $k+1 = n$, which we denote by $\widetilde{\phi}$. The approximation $\widetilde{\phi}$ satisfies the equation $F_\alpha(\phi) = O(x^{n-2})$. 

To solve the equation 
\begin{equation} \label{eqn: approximation, F of order at least n-2}
    F_\alpha(\phi) = O(x^{\nu}), \text{ with } \nu > n-2,
\end{equation}
we seek $h^{(k+1)}|_{x=0}$ for $k+1 \geq n$. 
Equation (2.7) in \cite{Alb20} says 
\begin{align} \label{eqn: k+1-th derivative of F}
    &(dx \otimes dx) \owedge \frac{1}{2} \big(A_1(\alpha) + 2n A_2(\alpha) \big) \ctr[h](h') \notag\\
    + \frac{1}{2} &\Big( (n-1)A_1(\alpha)h' + \big(A_1(\alpha) + 2n A_2(\alpha) \big) h \ctr[h](h') \Big) + \frac{1}{2}\lambda(\alpha) \ctr[h](h') \big(dx \otimes dx + h\big) \notag\\
    + x & (dx \otimes dx) \owedge \left(-\frac{1}{2} \big(A_1(\alpha) + 2A_2(\alpha) \big) \ctr[h](h'') + 2 A_2(\alpha) \ctr[h]^2(\Rm_h) \right)\\
    +x &A_1(\alpha) \cS\Big((dx \otimes 1) \owedge \ctr[h](Dh')\Big) \notag\\
    +x & \left(-\frac{1}{2}\right) \Big( A_1(\alpha)h'' + 2A_2(\alpha) h \ctr[h](h'')\Big) + x \Big( A_1(\alpha) \ctr[h](\Rm_h)+ 2A_2(\alpha) h \ctr[h]^2(\Rm_h)\Big) \notag\\
    +x &\Big( - \frac{1}{2n}\lambda(\alpha) \ctr[h](h'') + A_3(\alpha) \ctr[h]^2(\Rm_h) \Big) \big(dx \otimes dx + h\big)  \hspace{4cm} = O(x^3), \notag
\end{align}
where the constant $A_3(\alpha)$ is given by
\begin{align*}
    A_3(\alpha) &= \sum_q \alpha_q \left( -\frac{1}{2} \right)^q \frac{(n-2)!\; (2q-1)!\; (nq-1)}{(n-2q+1)!}.
\end{align*}
After multiplying by $x$, the on-diagonal terms on the left-hand side preserve the parity in $x$. This preservation extends to higher-order derivatives of the left-hand side with respect to $x$ and, consequently, carries over to Equations \eqref{eqn: smooth solution of lovelock eqn} and \eqref{eqn: smooth solution of lovelock eqn 2}. In the case where $n$ is odd, this implies that the right-hand side of Equation \eqref{eqn: smooth solution of lovelock eqn} vanishes. Thus, $\ctr[h](h^{(n)})|_{x=0} = 0$, but no constraint is imposed on the trace-free part of the expression. When $n$ is even, however, the right-hand side of Equation~\eqref{eqn: smooth solution of lovelock eqn 2} may have a non-vanishing trace-free part, whereas the left-hand side is purely trace. This is because the first term,
\[\frac{1}{2}(n-1-k) \, A_1(\alpha)\, h^{(k+1)}\]
vanishes when $k=n-1$.

To resolve this issue, we follow \cite[p.5]{Gra00} and introduce a logarithmic term $h_{n,1} x^n \log x$ in the expansion for $h$, namely,
\begin{equation} \label{eqn: modified expansion of h}
    h = h_0 + h_2 x^2 + \cdots + h_{n-2} x^{n-2} + h_{n,1} x^n \log x + h_n x^n + \cdots.
\end{equation}
Analogous to the derivation of Equations \eqref{eqn: smooth solution of lovelock eqn} and \eqref{eqn: smooth solution of lovelock eqn 2}, we examine Equation~\eqref{eqn: k+1-th derivative of F} by isolating the coefficients of $dx \otimes dx$ and the terms independent of $dx$. These coefficients are given explicitly as follows.
The coefficient of $dx \otimes dx$ is
\begin{align}
    I &= \frac{1}{2} \Big( A_1(\alpha) + 2n A_2(\alpha) \Big) \ctr[h](h') + \frac{1}{2}\lambda(\alpha) \ctr[h](h') - x A_2(\alpha) \ctr[h](h'') - x \frac{1}{2n} \lambda(\alpha) \ctr[h](h'') \notag\\
    &= \frac{1}{2} A_1(\alpha) \ctr[h](h') + B_{1,2}(\alpha) \ctr[h]\big( nh'-xh'' \big), \label{eqn: F_g dx^2}
\end{align}
while the coefficient of the terms without $dx$ is
\begin{align}
    II =&\; \frac{1}{2} (n-1) A_1(\alpha) h' + \frac{1}{2} \Big( A_1(\alpha) + 2n A_2(\alpha) \Big) h \ctr[h](h') +  \frac{1}{2}\lambda(\alpha) h \ctr[h](h') \notag\\
    &- \frac{1}{2} x \Big( A_1(\alpha)h'' + 2 A_2(\alpha)h \ctr[h](h'') \Big) - x \frac{1}{2n} \lambda(\alpha) h \ctr[h](h'') \notag\\
    =&\; \frac{1}{2} A_1(\alpha) \Big( (n-1)h'-xh'' \Big) + \frac{1}{2} A_1(\alpha) h\ctr[h](h') + B_{1,2}(\alpha) h\ctr[h]\big( nh'-xh'' \big). \label{eqn: F_g without dx}
\end{align}
Differentiating Equations \eqref{eqn: F_g dx^2} and~\eqref{eqn: F_g without dx} $(n-1)$ times with respect to $x$, we obtain the analogues of Equations \eqref{eqn: smooth solution of lovelock eqn} and \eqref{eqn: smooth solution of lovelock eqn 2} under the modified expansion \eqref{eqn: modified expansion of h}. In this case, the derivative $\pd_x^{n-1} II$ has a non-vanishing trace-free part given by
\begin{align*}
    &\; \pd_x^{n-1} \bigg( \frac{1}{2} A_1(\alpha) \Big( (n-1)h'-xh'' \Big)\bigg) \\
    =&\; \frac{1}{2} A_1(\alpha) \pd_x^{n-1} \bigg((n-1) h_{n,1} (n \log x + 1) x^{n-1} + n h_n x^{n-1} \\
    & \hspace{14ex}- x \Big( n(n-1)h_{n,1} \big(\log x + (2n-1)\big) + n(n-1) h_n \Big)x^{n-2} \bigg) + O(x)\\
    =&\; -\frac{n!}{2} A_1(\alpha) \, h_{n,1} + O(x) = \text{ terms involving } \pd_x^s h \text{ with } s < n.
\end{align*}

This procedure allows us to solve for $h_{n,1}$ in terms of $\pd_x^k \big|_{x=0} h$, where $k < n$. The introduction of the logarithmic term resolves the degeneracy at order $k = n-1$ in Equation~\eqref{eqn: k+1-th derivative of F}, thus allowing the equation to be satisfied without imposing additional constraints on $h_n$. This gives us the freedom to choose $h_n$. 

For $k+1>n$, the iteration proceeds as before by setting $h_n = 0$ and differentiating Equation \eqref{eqn: k+1-th derivative of F} $k$ times. In constructing the first approximation, it suffices to truncate the expansion at $k=n$, at which the on-diagonal components of $F_\alpha(g)$ are of order $O(x^{\nu})$ for some $\nu>n-2$. Moreover, by Lemma 2.1 in \cite{Alb20}, the off-diagonal components $F_\alpha(g)$ are of order $O(x^{n-1})$.

Thus, regardless of whether $n$ is even or odd, we can construct an approximation $\phibar$ satisfying the following conditions:
\begin{enumerate}
    \item $\phibar$ agrees with $\widetilde{\phi}$ up to order $n-2$.
    \item if $n$ is odd $\phibar^{(n)}|_{x=0}$ is trace free. 
    \item $\phibar$ does not include any terms of order $k \geq n$.
\end{enumerate}
 By construction, this approximation satisfies $F_\alpha(\phibar) = O(x^\nu)$, for some $\nu > n-2$. Therefore, we have $F_\alpha(\phibar) \in \cA_{\nu,\nu} \subset \cA_{\nu,\zeta'}$ for $\zeta' = \min\{\nu, \zeta\}$. 

If necessary, we apply Lemma \ref{lem: choice of gauge} and Remark \ref{rmk: gauge condition} to ensure that the gauge condition is satisfied.
\end{proof}

\begin{rmk}
    For simplicity, we will assume $\zeta \leq \nu$, that is, $\zeta = \zeta' = \min\{\nu, \zeta\}$.
\end{rmk}

\begin{defn}
    Let $L$ be the linearization of the modified Lovelock tensor $Q_\alpha(\cdot,\phibar)$ at the metric $\phi_0$ with respect to the first argument. We define $\mathbb{N}_L$ to be the monoid generated by the element $1$ and the upper indicial roots $\mu_+^{(i)}$ of $L$. An \bfemph{$L$-polyhomogeneous} expansion is a polyhomogeneous expansion of the form
    \[\sum_{k}\sum_{z \in \mathbb{N}_L} a_{z,k}(y) \,x^z\,(\log x)^k \]
    where $z \in \mathbb{N}_L$ and with no additional requirements on $k$ beyond those imposed by polyhomogeneity.
\end{defn}

\begin{thm} \label{main thm: phg of semi-ball} 
    In the same setting as Lemma \ref{lem: AHE implies smooth difference}, let $L$ be the linearization of the modified Lovelock tensor $Q_\alpha(\cdot,\phibar)$ at the metric $\phi_0$ with respect to the first argument. We take the sequence of weights $\{\sweight{k}\}$ where 
    \[a_0 = 0 \text{ and } \sweight{k+1} = \min \{ \lambda \in \mathbb{N}_L: \lambda > \sweight{k} \}.\] 
    There exists a sequence $\{\phi_k\}$ of $L$-polyhomogeneous metrics approximating $\phi$ that satisfies:  
    \begin{enumerate}
        \item \textit{(approximation)}. $\phi_k = \phibar + \psi_0 + \cdots + \psi_k$ with $\psi_j \in \Adot[\mu_+ + a_j]_\zeta(Y)$ for all $j \in \mathbb{N}$; 
        \item \textit{(polyhomogeneity)}. $Q_\alpha(\phi_k, \phibar) \in \cA_{\mu_+ + a_{k+1}, \zeta} (Y)$ for all $k \in \mathbb{N}$;
        \item \textit{(decreasing in transverse derivatives)}. Let $r_k = \phi - \phi_k$, for all $k, p \in \mathbb{N}$, there exists some $\delta > 0$ such that 
        \[\big( \nablatt \big)^p(r_k) \in C^0_{\sweight{k} + \delta, \zeta}(Y).\]
    \end{enumerate}
\end{thm}

The above theorem implies the regularity of a manifold with boundary diffeomorphic to a semi-ball since the Lovelock tensor and the gauge condition are diffeomorphism invariant. More precisely, we have the following corollary.
\begin{coro}\label{thm: phg of semi-ball}
    Let $X$ be a manifold with boundary. Let $M$ be the boundary defined by the boundary defining function $x$, equipped with a smooth metric $h$. Let $g_0 = \frac{dx^2 + h}{x^2}$ on the collar neighborhood $M_{\leq x_0}$ and suppose $g$ is an asymptotically hyperbolic Lovelock metric on $X$, that is,
    \[g - g_0 \in C^{2,\gamma}_\zeta(M_{\leq x_0}) \text{ for some } \zeta > 0, \gamma \in (0,1) \text{ and } F_\alpha(g) = 0.\] 
    Suppose $\LimSec(\alpha) \neq \emptyset$. Then for each point on $M$, there is a neighborhood $Y \subset X$, a diffeomorphism $\Phi: \overline{Y} \to B^{n+1}_+$ which is the identity on $\pd_\infty Y$ and a sequence $\{g_k\}$ of asymptotically hyperbolic metrics on $Y$ with finite polyhomogeneous expansion and a pair of positive numbers $(\delta, \eta)$ such that 
    \[\Phi^*g - g_k \in C^\infty_{\sweight{k} + \delta, \eta} (Y), \forall k \in \mathbb{N},\]
    where $a_k$ are defined in Theorem \ref{main thm: phg of semi-ball}. Moreover, the transverse derivatives satisfy the same estimates as in Theorem \ref{main thm: phg of semi-ball}.
\end{coro}

\subsection{Inductive argument}
In this section, we prove Theorem \ref{main thm: phg of semi-ball} by controlling the behavior of the polyhomogeneous expansion through the sequence of weights $a_k$, which goes to infinity.

We use an inductive argument on $k$ to prove the theorem. The idea hinges on a decomposition of  $Q_\alpha(\phi, \phibar)$ that separates the linear part of the modified Lovelock tensor from the higher-order terms, making the term involving the Laplace-type operator more apparent. Since $\phi$ is a Lovelock metric and $\phibar$ satisfies the gauge condition, $Q_\alpha(\phi, \phibar) = 0$. The decomposition of $Q_\alpha(\phi, \phibar)$ at a symmetric 2-tensor $\phi_k$ that approximates $\phi$ is given by
\begin{align} \label{eqn: Taylor expansion of Q}
    \begin{split}
        0 &= Q_\alpha(\phi, \phibar)\\
        &= Q_\alpha(\phi_k, \phibar) + D_1Q_\alpha(\phi_k, \phibar)(\phi_k - \phi) + q(\phi_k - \phi) \\
        &= Q_\alpha(\phi_k, \phibar) + L(\phi_k - \phi) + (D_1Q_\alpha(\phi_k, \phibar)-L)(\phi_k - \phi) + q(\phi_k - \phi), 
    \end{split}
\end{align}
where $q$ denotes the terms that are at least quadratic.

\begin{proof}[Proof of Theorem \ref{main thm: phg of semi-ball}] We denote by $x = x_0$ the boundary defining function of the boundary at infinity $M$ of $B^{n+1}_+$. Within a collar neighborhood of $M$, the asymptotically hyperbolic Lovelock metric $\phi$ satisfies 
\[\phi - \phi_0 \in C^{2,\gamma}_\zeta \text{ for some } \zeta > 0, \gamma \in (0,1) \text{ and } F_\alpha(\phi) = 0.\]
Note that $C^{2,\gamma}_\zeta \subset C^{2,\gamma}_{\zeta,\zeta}$, so $\phi - \phi_0 \in C^{2,\gamma}_{\zeta,\zeta}$. We will fix $\zeta$ in the proof.

\newpage
\textbf{Base case $k=-1$}. \emph{Assertion (i) and (ii)}. 

The first approximation is provided by the metric $\phibar$ constructed in Theorem \ref{thm: first approximation}. Given the fact that $a_0 = 0$, the strict inequality $\nu > \mu_+$ ensures the existence of some $\delta > 0$ such that $\nu = \mu_+ + a_0 + \delta = \mu_+ + \delta > \mu_+$.
Consequently,
\[Q_\alpha(\phi_{-1}, \overline{\phi}) = Q_\alpha(\overline{\phi}, \overline{\phi}) = F_\alpha(\overline{\phi}) \in \mathcal{A}_{\mu_+ + a_0 + \delta, \zeta}, \quad \delta > 0.\]  

\emph{Assertion (iii)}. It remains to check the decay in transverse derivatives using Proposition~\ref{prop: BH14, decreasing in transverse derivatives}. To apply the proposition, we check $r_{-1} = \phi - \phibar$ belongs to the appropriate H\"older space. 

By construction $r_{-1} = O(x^\nu)$ for $\nu > \mu_+$. Thus there exists some $\delta > 0$ such that $r_{-1}~\in~C^{\infty}_{\mu_+ + a_{-1} + \delta, \zeta}$. Then the decomposition of $Q_\alpha(\phi,\phibar)$ yields
\[L(r_{-1}) = -Q_\alpha(\phi, \phibar) - (D_1Q_\alpha(\phibar, \phibar)-L)(r_{-1}) - q_{-1}(r_{-1}).\]
The operator $\ell = D_1Q_\alpha(\phibar, \phibar)-L$ is in $x \cQ_1$ since $\phibar$ and $\phi$ are both asymptotically hyperbolic and the leading-order terms of the Laplacians $\Delta_\phibar$ and $\Delta_{\phi_0}$ agree to leading order. So Proposition \ref{prop: BH14, decreasing in transverse derivatives} implies 
\[ Q(\nablatt)^p (r_{-1}) \in C^0_{\mu_+ - 1 + \delta, \zeta}, \quad \text{for all } Q \in \cQ.\] 

We set $\phi_{-1} = \phibar$, $\psi_{-1} = 0$, $r_{-1} = \phi - \phibar$, $a_{-1} = -1$, and $a_0 = 0$.

\textbf{Induction}.
Suppose assertions (i)--(iii) hold for $k$. We will construct the next approximating metric $\phi_{k+1}$ using the Green's integral operators constructed in \S\ref{sec: Green's operator}. The expansion of the modified Lovelock tensor at $\phi_k$ is given by
\begin{equation} \label{eqn: decomposition of Q}
    0 = Q_\alpha(\phi, \phibar) = Q_\alpha(\phi_k, \phibar) + L(r_k) + (D_1Q_\alpha(\phi_k, \phibar)-L)(r_k) + q_k(r_k),
\end{equation}
where $r_k = \phi - \phi_k$. The construction of $\phi_{k+1}$ reduces to determining the $(k+1)$-st correction term $\psi_{k+1} = \phi_{k+1} - \phi_k$. After isolating the indicial operator from the decomposition, we can apply its inverse, represented by Green's integral operators, to construct the correction term.

We begin by rearranging Equation \eqref{eqn: decomposition of Q}: 
\[L(r_k) = -Q_\alpha(\phi_k, \phibar) - (D_1Q_\alpha(\phi_k, \phibar)-L)(r_k) - q_k(r_k).\]
This gives:
\begin{equation} \label{eqn: I(L)(r_k)}
    \sI(L)(r_k) = (L- L^\trans)(r_k) = -[Q_\alpha(\phi_k, \phibar)]_{\sweight{k+1}} - e_{k+1},
\end{equation}
where 
\begin{align} \label{eqn: e_k+1}
    \begin{split}
        e_{k+1} = &\; Q_\alpha(\phi_k, \phibar) - [Q_\alpha(\phi_k, \phibar)]_{\sweight{k+1}} \\
        &+ L^\trans (r_k) + \Big( D_1Q_\alpha(\phi_k, \phibar)-L \Big) (r_k) + q_k (r_k).
    \end{split}
\end{align}
We divide this construction into several cases based on the magnitude of the weight $\mu_+ + a_{k+1}$ (see Proposition \ref{prop: choice of G} and Proposition \ref{prop: extension of Green's operator to tensor bundles}), which determines the appropriate choice of the right inverse $\cG$ of $-\sI(L)$ in each case. Generally, if the first weight exceeds a certain threshold, $\cG|_{E_{\mu}} = G_{\infty}$ on the corresponding eigenspace $E_{\mu}$ (see Equation \eqref{eqn: Green's operator, G_infty}); otherwise, we take $\cG|_{E_{\mu}} = G_{0}$ (see Equation \eqref{eqn: Green's operator, G_0}). Under this choice of $\cG$, the right-hand side of Equation \eqref{eqn: I(L)(r_k)} can be rewritten as
\[\sI(L) \cG \Big( [Q_\alpha(\phi_k, \phibar)]_{\sweight{k+1}} + e_{k+1} \Big).\]
So the difference
\begin{equation} \label{eqn: ker I(L)(r_k)}
    r_k - \cG([Q_\alpha(\phi_k, \phibar)]_{\sweight{k+1}} + e_{k+1}) \in \ker \sI(L),
\end{equation}
which implies that it can be expressed as a linear combination of $x^{\mu^{(i)}_+}$ and $x^{\mu^{(i)}_-}$, where $\mu^{(i)}_\pm$ are indicial roots of $\sI(L)$.

\emph{Case 1: $\sweight{k} \geq \mu_+^{\max}$}. By the inductive hypothesis, $r_k \in C^0_{\sweight{k} + \delta, \zeta}$ for some $\delta > 0$. We choose $\cG|_{E_{\mu}} = G_\infty$ for all eigenvalues $\mu$. Since $\sweight{k} + \delta > \mu_+^{\max}$, Proposition \ref{prop: choice of G} guarantees that $\cG$ is the right inverse of $-\sI(L)$. Furthermore, since $\sweight{k} \geq \mu_+^{\max}$, neither of the terms $x^{\mu_+^{(i)}}$ nor $x^{\mu_-^{(i)}}$ lies in $\ker \sI(L)$.
 By Equation \eqref{eqn: ker I(L)(r_k)}, the error term $r_k$ admits a decomposition
\begin{align*}
    r_k = \cG \Big( [Q_\alpha(\phi_k, \phibar)]_{\sweight{k+1}} \Big) + \cG(e_{k+1}).
\end{align*}
We then define 
\[\psi_{k+1} = \cG \Big( [Q_\alpha(\phi_k, \phibar)]_{\sweight{k+1}} \Big) \in \Adot[\sweight{k+1}]_\zeta \quad \text{and} \quad r_{k+1} = \cG(e_{k+1}).\]

To check assertion (ii), we decompose $Q_\alpha(\phi_{k+1}, \phibar) $ as follows
\begin{align*}
    Q_\alpha(\phi_{k+1}, \phibar) 
    =\;& [Q_\alpha(\phi_k, \phibar)]_{\sweight{k+1}} 
    + \sI(L_g)(\psi_{k+1}) \tag{I} \\
    &+ \big(Q_\alpha(\phi_k, \phibar) 
    - [Q_\alpha(\phi_k, \phibar)]_{\sweight{k+1}}\big) \tag{II} \\
    &+ L^{\trans}(\psi_{k+1}) 
    + \big(D_1 Q_\alpha(\phi_k, \phibar) - L\big)(\psi_{k+1}) 
    + q_k(\psi_{k+1}). \tag{III}
\end{align*}

\begin{itemize}
    \item (I) vanishes by definition of $\psi_{k+1}$. Indeed, since $-\sI(L_g)\cG = \id $
    \begin{align*}
        \mathrm{(I)} &= [Q_\alpha(\phi_k, \phibar)]_{\sweight{k+1}} + \sI(L_g) \cG \Big( [Q_\alpha(\phi_k, \phibar)]_{\sweight{k+1}} \Big)\\
        &= [Q_\alpha(\phi_k, \phibar)]_{\sweight{k+1}} - [Q_\alpha(\phi_k, \phibar)]_{\sweight{k+1}} = 0.
    \end{align*}
    
    \item (II) belongs to $\cA_{\sweight{k+2},\zeta}$ by the definition of $[~]_{\sweight{k+1}}$ and $a_{k+2}$. 
    
    \item It remains to show (III) belongs to $\cA_{\sweight{k+2},\zeta}$. Note that transverse derivatives increase the first weight by at least 1 and radial derivatives preserve the first weight. This takes care of the first term in (III). For the second term, recall we know $D_1Q_\alpha(\phi_k, \phibar)-L \in x \cQ_1$ for asymptotically hyperbolic $\phi_k$, so the first weight gets amplified by at least 1. Finally the at least quadratic term doubles the first weight. So line (III) belongs to $\cA_{\sweight{k+2},\zeta}$. 
\end{itemize}

We pass the decay for the transverse derivatives of $r_{k+1}$ to the decay for the transverse derivatives of $e_{k+1}$ via commutativity with $\cG$, under the condition $\sweight{k+1}+\delta > \mu_+^{\max}$:
\[(\nablatt)^p r_{k+1} = (\nablatt)^p \Big( \cG(e_{k+1}) \Big) = \cG \Big( (\nablatt)^p e_{k+1} \Big).\]
Indeed, this follows from \cite[p. 846]{BH14}, for $u \in C^0_\xi$, 
\begin{align*}
    \left\vert x^{-\theta} \int_0^x \widetilde{x}^{-\theta-1} \, u(\widetilde{x}) \, d\widetilde{x} \right\vert &\leq C x^{\xi}, \text{ when } \theta < \xi.
\end{align*}

\begin{align*}
    \left\vert x^{-\xi} \int_x^{x_0} \widetilde{x}^{-\xi-1} \, u(\widetilde{x}) \, d\widetilde{x} \right\vert &\leq C x^{\xi}, \text{ when } \theta = \xi.
\end{align*}
So $u(x)$ and the transverse derivative of $\cG(u(x))$ have the same decay rate as $x \to 0$.

It remains to show that the transverse derivatives of 
\begin{align*}
    e_{k+1} =&\; Q_\alpha(\phi_k, \phibar) - [Q_\alpha(\phi_k, \phibar)]_{\sweight{k+1}} \\
    &+ L^{\trans}(r_k) + \Big(D_1Q_\alpha(\phi_k, \phibar) - L\Big)(r_k) + q_k(r_k)
\end{align*}
belong to $C^0_{\sweight{k+1} + \delta, \zeta}$. We check the following:
\begin{itemize}
    \item By $L$-polyhomogeneity, all transverse derivatives of $Q_\alpha(\phi_k, \phibar) - [Q_\alpha(\phi_k, \phibar)]_{\sweight{k+1}}$ belong to $\cA_{\mu_+ + a_{k+2}, \zeta}$. 
    
    \item By the inductive hypothesis on $r_k$, we know $L^{\trans}(r_k) \in C^0_{\sweight{k}+1 + \delta, \zeta}$. 
    
    \item By the inductive hypothesis on $r_k$ and by the fact that $D_1Q_\alpha(\phi_k, \phibar) - L \in x \cQ_1$, we know the transverse derivative of $\Big(D_1Q_\alpha(\phi_k, \phibar) - L\Big)(r_k) \in C^0_{\sweight{k} + 1 + \delta, \zeta}$. 

    \item Finally, the transverse derivatives of $q_k(r_k)\in C^0_{\sweight{k+1} + \delta, \zeta}$ because $q_k$ is the at least quadratic term and since $\mu_+ \geq 1$.
\end{itemize}
We conclude $(\nablatt)^p e_{k+1} \in C^0_{\sweight{k+1} + \delta, \zeta}$.

\emph{Case 2: $\sweight{k} < \mu_+^{\max}$}.
This case further divides into two subcases due to the nature of Green’s integral operator in Equations \eqref{eqn: Green's operator, G_infty} and \eqref{eqn: Green's operator, G_0}. Specifically, as pointed out in Proposition \ref{prop: choice of G}, $G_0$ may not map $e_{k+1}$ to the desired H\"older space.

Recall that the indicial roots of $L$ are given in Example \ref{ex: laplacian is an isom} and are denoted by $\mu_\pm^{(i)}$, with~$i = 0,1,2,3$.

\emph{Subcase 2.1: for all $i$, $\mu_+^{(i)} \notin [\sweight{k}+\delta, \sweight{k+1}]$}.

\begin{center}
    \includegraphics[width=0.4\textwidth]{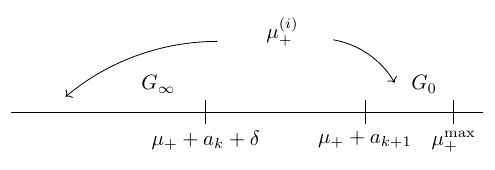}
\end{center}  
Define the index set $\cI = \{i: \sweight{k+1} < \mu_+^{(i)}\}$. By the inductive hypothesis, we know that~$r_k \in C^0_{\sweight{k} + \delta, \zeta}$. Since $\ker \sI(L) \cap C^0_{\sweight{k} + \delta, \zeta}$ is spanned by $x^{\mu_+^{(i)}}$ for $i \in \cI$, Equation~\eqref{eqn: ker I(L)(r_k)} implies that $r_k$ decomposes as
\begin{equation} \label{eqn: r_k}
    r_k = \cG \Big( [Q_\alpha(\phi_k, \phibar)]_{\sweight{k+1}} \Big) + \cG(e_{k+1}) + A(x) 
\end{equation}
where $e_{k+1}$ is given by Equation \eqref{eqn: e_k+1},
\begin{align*}
    \cG\big|_{E_{\mu_+^{(i)}}} = \begin{cases}
        G_0 & \text{if } i \in \cI \\
        G_\infty & \text{otherwise},
    \end{cases}
\end{align*}
and $A(x) = \sum_{i\in\cI} A_i \,  x^{\mu_+^{(i)}}$,
with $A_i$ sections at the boundary at infinity. We define
\[\psi_{k+1} = \cG \Big( [Q_\alpha(\phi_k, \phibar)]_{\sweight{k+1}} \Big) \in \Adot[\sweight{k+1}]_\zeta.\]
Then $r_{k+1} = \phi - \phi_{k+1}  =r_k - \psi_{k+1} = \cG(e_{k+1}) + A(x)$.

The only difference from the previous case in the proof lies in the additional term $A(x)$. Since $\psi_{k+1}$ remains unchanged, this additional term does not affect the decomposition (I)–(III). Therefore, assertion (ii) holds for the same reason as in Case 1. To complete the proof, we now focus on the transverse derivatives of $r_{k+1}$. On the eigenspaces where~$\cG = G_\infty$, the proof is identical to that of Case 1. It remains to consider the case where $\cG = G_0$.
\begin{equation} \label{eqn: r_k+1}
    r_{k+1} = r_{k} - \psi_{k+1} = G_0\Big(e_{k+1} \Big) + \sum_{i\in\cI} A_i \,  x^{\mu_+^{(i)}}, 
\end{equation}
where $e_{k+1}$ is the same as in the previous case, as seen in Equation \eqref{eqn: e_k+1}.
\begin{itemize}
    \item We claim that $A_i$ is smooth with respect to transverse directions. This immediately implies 
    \[\big( \nablatt \big)^p \Big(A_i \,x^{\mu_+^{(i)}}\Big) \in \Adot_{\mu_+^{(i)}, \zeta} \subset \Adot_{\sweight{k+1}+\delta, \zeta},\]
    for $0 < \delta \leq \mu_+^{(i)} - \sweight{k+1}$. To justify the claim we write using Equation \eqref{eqn: r_k} that 
    \[A_i = x^{-\mu_+^{(i)}} \Big[ r_{k} - \psi_{k+1} - G_0(e_{k+1}) \Big].\]
    A priori, $A_i$ depends only on transverse directions, so we restrict our attention to the slice~$x = x_0$. 
    By the inductive hypothesis on $r_k$, $r_{k+1}$ and $\psi_{k+1}$ are smooth in the transverse variables $(x^1, \cdots, x^n)$. The term $G_0(e_{k+1})$ is also smooth in transverse variables, so the conclusion follows.
    
    \item Again, we pass the decay for the transverse derivatives of $r_{k+1}$ to the decay for the transverse derivatives of $e_{k+1}$ via commutativity with $\cG$:
    \[ (\nablatt)^p \Big( G_0(e_{k+1}) \Big) = G_0 \Big( (\nablatt)^p e_{k+1} \Big), \]
    using the integral control given in 
    \cite[p.844]{BH14} with the pair of real numbers $\alpha_\pm = \mu_\pm^{(i)}$. So transverse derivatives of the second term of Equation \eqref{eqn: r_k+1} belongs to $\Adot_{\mu_+^{(i)}, \zeta}$. 
\end{itemize}

\emph{Subcase 2.2: for some $i$, $\mu_+^{(i)} \in [\sweight{k}+\delta, \sweight{k+1}]$}.
\begin{center}
        \includegraphics[width=0.4\textwidth]{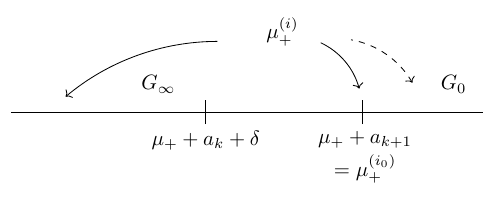}
\end{center}
By definition of the sequence $\sweight{k}$, we know $\sweight{k+1} = \mu_+^{(i_0)}$ for some $i_0$, and there is no critical weight other than this $\mu_+^{(i_0)}$ belonging to the interval.

Similar to Subcase 2.1, the decomposition of $r_k$ includes polynomials in $x$ that lie in the kernel of $\sI(L_g)$ , as given in Equation \eqref{eqn: ker I(L)(r_k)}. Moreover, a correction term $A_0 x^{\sweight{k+1}} \log x$ 
appears because $G_0$ does not necessarily map $\Adot[\sweight{k+1}]_{\zeta}$ to itself when $\sweight{k} = \mu_+^{(i)}$ (see Proposition \ref{prop: choice of G}). It follows that $r_k$ decomposes as 
\begin{align*}
    r_k = \cG \Big( [Q_\alpha(\phi_k, \phibar)]_{\sweight{k+1}} \Big) + \cG(e_{k+1}) + A(x) + A_0 \,x^{\sweight{k+1}} \log x \in C^0_{\sweight{k} + \delta, \zeta},
\end{align*}
where $A(x) = \sum_{i\in\cI} A_i \,  x^{\mu_+^{(i)}}$ and \begin{align*}
    \cG\big|_{E_{\mu_+^{(i)}}} = \begin{cases}
        G_0 & \text{if } \sweight{k+1} + \delta \leq \mu_+^{(i)}\\
        G_\infty & \text{otherwise}.
    \end{cases}
\end{align*}
Thus we set 
\[\psi_{k+1} = \cG \Big( [Q_\alpha(\phi_k, \phibar)]_{\sweight{k+1}} \Big) + A_0x^{\sweight{k+1}}\log x \in \Adot[\sweight{k+1}]_\zeta.\]
Then $r_{k+1} = \phi - \phi_k = r_k - \psi_{k+1} = \cG(e_{k+1}) + A(x)$.

We focus on the transverse regularity of the new term involving $A_0$, while the remaining terms have been discussed in previous cases. As with the proof for $A(x)$ in the previous case, we claim that $A_0$ is smooth with respect to the transverse directions. We can again restrict to the slice $x=x_0$ since $A_0$ depends only on transverse variables. Then the smoothness of~$r_{k+1}$ and $r_k$ implies the claim.

Now assertions (i)--(iii) follow immediately from the fact that
\[\big( \nablatt \big)^p \Big(A_0\, x^{\sweight{k+1}} \log x\Big) \in \Adot_{\sweight{k+1}, \zeta} \subset \Adot_{\sweight{k}+\delta, \zeta},\]
for $0 < \delta \leq \mu_+^{(i)} - (\sweight{k+1})$.
\end{proof}

\section{The ambient obstruction tensor} \label{sec: ambient obstruction}

We investigate the obstruction to the existence of a smooth formal power series solution for a conformally compact Lovelock metric with a prescribed conformal infinity. Graham and Hirachi \cite{GH04} have previously explored the properties of such obstructions in the context of conformally compact Einstein metrics.

Corollary \ref{thm: phg of semi-ball}, together with \cite{Alb20}, implies that a conformally compact Lovelock metric~$g = \frac{dx^2 + h}{x^2}$ on a manifold $X^{n+1}$, analogous to its Einstein counterpart, admits a Fefferman-Graham expansion: 
\begin{align*}
    h_x = \begin{cases}
        h_0 + h_2 x^2 + \text{(even powers)} + h_{n-1} x^{n-1} + h_n x^n + \cdots & n \text{ odd}\\
        h_0 + h_2 x^2 + \text{(even powers)} + h_{n,1} \log (x)x^{n-1} + h_n x^n + \cdots  & n \text{ even}.
    \end{cases}
\end{align*}
The expansion is smooth when $n$ is odd. For even $n$, smoothness is obstructed by the term~$h_{n,1} \log(x)x^{n-1}$. To extract the logarithmic coefficient from the expansion, we define the tensor $\sO$.
\begin{defn}
    Recall the Lovelock tensor is given by 
    \[F_\alpha(g) = \sum \alpha_q \left( \Ricqq_g - \lambda^\qq g \right) -\frac{\alpha_q}{2q} \left( \scalqq_g - (n+1)\lambda^\qq  \right)g.\]
    We define the (ambient) \bfemph{obstruction tensor} to be 
    \[\sO = x^{2-n} \tf\big(F_\alpha(g)\big)\big|_{x=0},\] 
    where $\tf$ denotes the trace-free part with respect to $g$. 
\end{defn}
Albin \cite[Lemma 2.1]{Alb20} showed that the obstruction tensor is trace-free and divergence-free. We will compute that the  leading order term of $\sO$ is essentially the same as that of the ambient tensor of a conformally compact Lovelock metric $g$.

\begin{thm}[Leading-Order Term of the Lovelock Obstruction Tensor]  \label{main thm: obstruction leading order}
    Let $M$ be a smooth $n$-manifold with conformal infinity $[h]$, where $n \geq 4$, and let $X$ be an $(n+1)$-manifold with boundary $M$. Let~$x$ denote a boundary defining function for $M$.

    There exists a metric $g$ satisfying $x^2 g|_{TM} \in [h]$ and $F_\alpha(g) = O(x^{n-2})$. The leading-order term of the obstruction tensor $\sO = x^{2-n} \tf\big(F_\alpha(g)\big)\big|_{x=0}$ is given by
    \begin{equation} \label{eqn: h.o.t. of O_ij}
        \frac{A_1(\alpha)}{c_n} \Delta^{\frac{n}{2} - 2} \left( \tensor{P}{_{ij,k}^k} - \tensor{P}{_k^k_{,ij}} \right),
    \end{equation}
    where $P$ is the Schouten tensor of $h$, and the coefficients are given by
    \begin{align*}
        A_1(\alpha) = \sum_q \alpha_q \left( -\frac{1}{2} \right)^{q-1} \frac{(n-2)!}{2} \frac{(2q-1)!}{(n-2q)!}, \quad \text{and} \quad c_n = \frac{2^{n-2}(n/2-1)!^2}{n-2}.
    \end{align*} 
\end{thm}

\begin{proof}
    The proof of the lemma is a direct computation. A priori, we have $F_\alpha(g) = O(x^{-1})$, so we work with $x F_\alpha(g)$. Note that the obstruction tensor is by definition the trace-free part, so we focus on the Ricci-$\qq$ terms in $F_\alpha(g)$.

    As shown in Equation (2.5) of \cite{Alb20}, the Ricci-$\qq$ curvatures of $h$ are determined by the boundary data of the conformally compact metric
    \begin{align} 
    \begin{split} \label{eqn: Ricci-2q}
        &\hspace{2ex}\Big[ \Ricqq_g \Big]_{ij} \\[3ex]
        =&\phantom{+} \Big[\ctr[h]^{2q-1} (\Rm_g^q)\Big]_{ij} \\[2ex]
        =& \phantom{+}\frac{1}{x^2} \ctr[h]^{2q-1} \Bigg\{ \bigg[ x^2 \Rm_h - \frac{1}{2} \Big(\frac{x}{2} h' - h \Big)^2  \bigg]^q \Bigg\} \\
        &+ \frac{1}{x^2} (2q-1)q \ctr[h]^{2q-2} \Bigg\{ \bigg[ \frac{x^2}{2} \Big(-h'' + \frac{1}{2} \ctr[h](h')h' - \frac{1}{4} \ctr[h]((h')^2) \Big) + \frac{x}{2}h' - h \bigg] \\
        &\cdot \bigg[ x^2 \Rm_h - \frac{1}{2} \Big(\frac{x}{2} h' - h \Big)^2  \bigg]^{q-1} \Bigg\}.
    \end{split}
    \end{align}

    We begin by computing the $s$-th order derivatives of the first term in $xF_\alpha(g)$ for $s > 1$. Using Equation \eqref{eqn: Ricci-2q}, and noting that $\lambda^\qq g$ cancels the $O(x^{-1})$ term in $\Ricqq_g$, we have
    \begin{align*}
        &\pd_x^s \Bdy x \left({\Ricqq_g}_{ij} - \lambda^\qq  g_{ij} \right)\\[2ex]
        = &-q \left( -\frac{1}{2} \right)^q \ctr[h]^{2q-1} [ h^{2q-1} \;\pd_x^{s+1} \bdy (h)] + qs \left( -\frac{1}{2} \right)^{q-1} \;\pd_x^{s-1} \bdy \ctr[h]^{2q-1} [ h^{2q-2} \Rm_h] \\
        &+ (2q-1)qs \left( -\frac{1}{2} \right)^q \ctr[h]^{2q-2} [ h^{2q-2} \;\pd_x^{s+1} \bdy (h)] - (2q-1)q \left( -\frac{1}{2} \right)^q \ctr[h]^{2q-2} [ h^{2q-2} \;\pd_x^{s+1} \bdy (h)] \\
        &+ (2q-2)(2q-1)qs \left( -\frac{1}{2} \right)^{q-1} \;\pd_x^{s-1} \bdy \ctr[h]^{2q-2} [ h^{2q-3} \Rm_h] \\
        &- (2q-2)(2q-1)q \left( -\frac{1}{2} \right)^q \ctr[h]^{2q-2} [ h^{2q-2} \;\pd_x^{s+1} \bdy (h)] + \text{terms involving } \pd_x^k h_{ij} \text{ with } k< s.
    \end{align*}
    The contractions on the right-hand side can be expressed in terms of $\Ric_h, \scal_h$, and derivatives of $h$ using Proposition \ref{prop: -1 contraction of Rm and g}. Moreover, since interchanging the derivative with the contraction operator does not affect the leading order term, we obtain 
    \begin{align*}
    \begin{split}
        \pd_x^s \Bdy x &\left({\Ricqq_g}_{ij} - \lambda^\qq  g_{ij} \right)\\[2ex]
        = &\phantom{+} q(-n+s+1) \left( -\frac{1}{2} \right)^q \frac{(n-2)!}{(n-2q)!} (2q-1)!  \;\pd_x^{s+1} \bdy (h) \\
        &+ qs \left( -\frac{1}{2} \right)^{q-1} \frac{(n-2)!}{(n-2q)!} (2q-1)! \;(-n+s+1) \; \pd_x^{s-1} \bdy (\Ric_h)\\
        &+ q[(n-1)(1-2q)+2s(q-1)] \left( -\frac{1}{2} \right)^q \frac{(n-2)!}{(n-2q+1)!} (2q-1)! \; \ctr[h] [\pd_x^{s+1} \bdy (h)]\, h\\
        &+ q(q-1)s \left( -\frac{1}{2} \right)^{q-1} \frac{(n-2)!}{(n-2q+1)!} (2q-1)!  \; \pd_x^{s-1} \bdy (\scal_h) \,h \\
        &+ \text{terms involving } \pd_x^k h_{ij} \text{ with } k< s.
    \end{split}
    \end{align*}

    By invoking power series in the expression of $F_\alpha(g)$, we can write the derivatives as
    \begin{align} \label{eqn: derivative of F}
        \begin{split}
            \pd_x^s \Bdy x \cdot \sum_q \alpha_q &\left({\Ricqq_g}_{ij} - \lambda^\qq  g_{ij} \right)\\ 
            = &\phantom{+} \frac{A_1(\alpha)}{2} \Big[ (n-s-1) \pd_x^{s+1} \bdy (h) + 2s \,\pd_x^{s-1} \bdy (\Ric_h) \Big] \\
            &+ \Big[(n-1-s) A_2(\alpha) - \frac{1}{2n}\lambda(\alpha)\Big] \; \ctr[h] [\pd_x^{s+1} \bdy (h)]\, h\\
            &+ A_2(\alpha) \; \pd_x^{s-1} \bdy (\scal_h) \,h +  \text{terms involving } \pd_x^k h_{ij} \text{ with } k< s,
        \end{split}
        \end{align}
    where the constants are the same as those defined in \cite[p.4, p.17]{Alb20}. For the same reason in \S\ref{sec: Lovelock tensor}, we assume $\kappa = 1$ for simplicity:
    \begin{align*} 
        A_1(\alpha) &= \sum_q \alpha_q \left( -\frac{1}{2} \right)^{q-1} \frac{(n-2)!}{2} \frac{(2q)!}{(n-2q)!}, \\
        A_2(\alpha) &= \sum_q \alpha_q \left( -\frac{1}{2} \right)^{q-1} \frac{(n-2)!}{2} \frac{(2q)!}{(n-2q+1)!} (q-1), \\
        \lambda(\alpha) &= \sum_q \alpha_q \left( -\frac{1}{2} \right)^q \frac{n!\; (2q)!}{(n-2q+1)!}.
    \end{align*}
    
    Observe that only the first term on the right-hand side of Equation \eqref{eqn: derivative of F} contributes to the leading order of the trace-free part of $x F_\alpha(g)$. Additionally, the behavior of the leading-order term in $\tf [\pd_x^{s+1} \bdy (h)]$ aligns with its Einstein counterpart, shown in \cite[Theorem~2.1]{GH04}, since no generalized Ricci curvature appears. More precisely, this leads to the same result for the highest-order term in the derivatives of $h$:
\begin{align*}
\begin{split}
    \tf \left[ \pd_x^{2\ell} \bdy (h) \right] 
    &= 2 \cdot \frac{3 \cdot 5 \cdots (2\ell - 1)}{(n - 4)(n - 6) \cdots (n - 2\ell)} \cdot \Delta^{\ell - 2} \bigg( 
        \tensor{P}{_k^k_{,ij}} 
        - \tensor{P}{_{ij,k}^k} 
    \bigg) \\
    &\quad + \text{terms involving } \pd_x^k h_{ij} 
        \text{ with } k < 2\ell.
\end{split}
\end{align*}
where $2 \leq \ell < n/2$. Differentiating $x \tf(F_\alpha(g))_{ij} = x^{n-1} \sO_{ij} \mod O(x^n)$ $(n-1)$-times and then restricting to the boundary yields
\begin{align*}
    &\tf \left[\pd_x^{n-1} \Bdy x \cdot \sum_q \alpha_q \left({\Ricqq_g}_{ij} - \lambda^\qq  g_{ij} \right)\right] \\
    =&A_1(\alpha) (n-1) \tf\left[ \pd_x^{n-2} \bdy (\Ric_h) \right] + \left(\text{terms involving } \pd_x^k h_{ij} \text{ with } k< n-1 \right) = (n-1)! \sO_{ij}.
\end{align*}
From this, we conclude that
\begin{alignat*}{2}
  \sO_{ij} 
  &=\frac{A_1(\alpha)}{(n-2)!} \,\frac{3 \cdot 5 \cdots (n-3)}{(n-4)(n-6) \cdots 2} 
  &&\hspace{-3ex}\Delta^{n/2 - 2} \left( \tensor{P}{_{ij,k}^k} - \tensor{P}{_k^k_{,ij}} \right) \\
  & 
  &&+ \text{terms involving } \pd_x^k h_{ij} \text{ with } k< n-1\\
  &= \frac{A_1(\alpha)}{c_n}  \,\Delta^{n/2 - 2} \left( \tensor{P}{_{ij,k}^k} - \tensor{P}{_k^k_{,ij}} \right)
  &&+ \text{terms involving } \pd_x^k h_{ij} \text{ with } k< n-1.
\end{alignat*}

We point out that $c_n$ appears in the denominator since the definition for $\sO$ here is different from Graham and Hirachi's definition \cite{GH04} by that constant.
\end{proof}

\begin{rmk}
    Theorem \ref{main thm: obstruction leading order}, in conjunction with Lemma 2.1 from \cite{Alb20}, generalizes Assertions (1)--(3) of Theorem 2.1 in \cite{GH04} from the Einstein case to the Lovelock case.
\end{rmk}

\section{Singular Yamabe-\texorpdfstring{$\qq$}{2q} problem}  \label{sec: singular Yamabe}
Just like for conformally compact Einstein metrics, where one solves the Einstein equation as a Taylor series at the boundary to the extent possible, one can formally solve the constant scalar curvature equation to the extent possible \cite[\S2]{Gra17}. Here, we extend the discussion to the Lovelock setting and refer to the problem of formally solving for constant scalar-$\qq$ curvatures as \bfemph{singular Yamabe-$\boldsymbol{\qq}$ problem.} 

Consider $(X,\gbar)$ a manifold with boundary of dimension $n+1$ and denote $\pd X = M$. Let $x$ denote the special boundary defining function, which represents the geodesic distance to $M$ with respect to $\gbar$. We aim at a conformal metric $g = u^{-2}\gbar$ that has constant scalar-$\qq$ curvature. Equivalently, we seek to find a positive function $u$ on $X$, which is a solution of the following equation:
\begin{equation}\label{eqn: sing Yamabe prob}
    \scal_g^\qq = (n+1) \lambda^\qq,
\end{equation}
where $\lambda^\qq = \big(-\frac{1}{2}\big)^q \frac{n!(2q)!}{(n-2q+1)!}$ are the constants previously defined in \S\ref{sec: Lovelock tensor}.

Let $\kappa = |du|^2_{\gbar}$. For a multi-index $\beta = (\beta_1, \dots, \beta_q)$ with $2q \leq n$, we define 
\[\widetilde{B}_{1,2}(\beta, \kappa) = A_1(\beta, \kappa) + (n+1)A_2(\beta, \kappa),\]
where 
\begin{align*} 
    A_1(\beta,\kappa) &= \sum_q \beta_q \left( -\frac{\kappa}{2} \right)^{q-1} \frac{(n-2)!}{2} \frac{(2q)!}{(n-2q)!}, \\
    A_2(\beta,\kappa) &= \sum_q \beta_q \left( -\frac{\kappa}{2} \right)^{q-1} \frac{(n-2)!}{2} \frac{(2q)!}{(n-2q+1)!} (q-1).
\end{align*}
Note that $\widetilde{B}_{1,2}(\beta, \kappa)$ coincides with the constant $B_{1,2}(\alpha, \beta, \kappa)$ defined in \cite{Alb20} when $\alpha = 0$. A straightforward simplification then yields the explicit formula:
\[\widetilde{B}_{1,2}(\beta, \kappa) = \sum_q \beta_q \left( -\frac{\kappa}{2} \right)^{q-1} \frac{(n-1)!}{2} \frac{(2q)! \; q}{(n-2q+1)!}.\]
We will see that 
\begin{thm}[Formal Polyhomogeneous Solution to the Singular Yamabe-$(2q)$ Problem]
    For a prescribed $\beta = (\beta_1, \cdots, \beta_q)$, with $2q\leq n$, satisfying $\widetilde{B}_{1,2}(\beta, \kappa) \neq 0$, there exists a function $u$ on $X$ that serves as a boundary defining function of $M = \pd X$ such that the scalar-$\qq$ curvatures of the conformally rescaled metric $g = u^{-2}\gbar$ satisfy
    \begin{equation}
        \widetilde{F}_\beta(g) = \sum \beta_q \Big(\scal_g^\qq - (n+1)\lambda^\qq\Big) = O(x^{n+2} \log x),
    \end{equation}
    where $\lambda^\qq = \big(-\frac{1}{2}\big)^q \frac{n!(2q)!}{(n-2q+1)!}$. 
    Moreover, the function $u$ has the form
    \[u = x + u_2 x^2 + \cdots + u_{n+1} \,x^{n+1} + \cL^\qq x^{n+2} \log x, \]
    where $\cL^\qq$ is the singular Yamabe-$(2q)$ obstruction.
\end{thm}

To solve Equation \eqref{eqn: sing Yamabe prob}, we begin by applying the conformal transformation of the Riemannian curvature tensor 
\[\Rm_{u^{-2} \gbar} = \frac{1}{u^4} \Big( u^2 \Rm_{\gbar} + \gbar \big( u \Hess u  - \frac{\kappa}{2}\gbar \big)\Big). \]
This allows us to express the conformal transformation of the scalar-$\qq$ curvature as follows:
\begin{align*}
    \scalqq_{u^{-2} \gbar} &= \ctr[{u^{-2} \gbar}]^\qq \Big(\Rm_{u^{-2} \gbar}^q\Big) \\
    &= (u^2)^{2q} \;\ctr[\gbar]^{2q} \bigg[ \Big(\frac{1}{u^4}\Big)^q \Big( u^2 \Rm_{\gbar} + \gbar \big( u \Hess u  - \frac{\kappa}{2}\gbar \big)\Big)^{q} \bigg]  \\
    &= \ctr[\gbar]^\qq \bigg[\Big( u^2 \Rm_{\gbar} + \gbar \big( u \Hess u  - \frac{\kappa}{2}\gbar \big)\Big)^q \bigg].
\end{align*}
For convenience, we introduce the notation $\widetilde{F}_q(g) = \scal_g^\qq - (n+1)\lambda^\qq$ and seek an approximate solution $u$ to Equation \eqref{eqn: sing Yamabe prob} by computing the coefficients 
$u_i$ in the formal expansion
\[u = u_1 x + u_2 x^2 + \cdots.\]
By an appropriate rescaling, we normalize the leading coefficient so that $u_1 = 1$. The remaining coefficients $u_i$ are determined inductively by differentiating $\widetilde{F}_q(g)$ successively. Specifically, for $s \geq 0$, we obtain the following: 
\begin{align} \label{eqn: scalar-2q derivative}
\begin{split}
    &\pd_x^s \Big(\widetilde{F}_q(g) \Big)\\
    =&\phantom{+} q \ctr[\gbar]^\qq \bigg[\Big( u^2 \Rm_{\gbar} + \gbar \big( u \Hess u  - \frac{\kappa}{2}\gbar \big)\Big)^{q-1}  \Big(s \pd_x u\, \gbar\,  \pd_x^{s-1} (\Hess u) -  \pd_x^{s+1} u\, \pd_x u\, \gbar^2 \Big)\bigg] \\
    &+ \text{terms involving } \pd_x^k u \text{ with } k\leq s.
\end{split}
\end{align}

For convenience, we temporarily denote
\[T = u^2 \Rm_{\gbar} + \gbar \big( u \Hess u  - \frac{\kappa}{2}\gbar \big) \in \Omega^{2\otimes 2}, \quad \eta = \pd_x^{s-1} \Hess u  \in \Omega^{1\otimes 1}. \]
Using the contraction formula \eqref{eqn: contraction formula}, Equation \eqref{eqn: scalar-2q derivative} becomes
\begin{align} \label{eqn: scalar-2q derivative 2}
\begin{split}
    \pd_x^s \Big(\widetilde{F}_q(g) \Big) =&\phantom{+} q \, \pd_x u \, \ctr[\gbar]^\qq \bigg[ T^{q-1}  \Big(s \gbar\eta -  \pd_x^{s+1} u\,\gbar^2 \Big)\bigg]\\
    =&\phantom{+} qs \pd_x u (n-2q+2)(2q) \ctr[\gbar]^{2q-1} \big(T^{q-1} \eta \big)\\
    &- q \pd_x u (n-2q+3)(n-2q+2)(2q)(2q-1) \ctr[\gbar]^{2q-2} \big(T^{q-1} \big) \pd_x^{s+1}u \\
    &+ \text{terms involving } \pd_x^k u \text{ with } k\leq s. 
\end{split}
\end{align}
Next, we focus on the contraction term $\ctr[\gbar]^{2q-1} \big(T^{q-1} \eta \big)$ appearing in the first term of Equation~\eqref{eqn: scalar-2q derivative 2}. 

A key observation based on the commutativity of $\ctr[\gbar]$ and $\pd_x$ up to lower order derivatives of $u$ is that
\[\ctr[\gbar](\eta) = \pd^{s+1}_x u + \text{terms involving } \pd_x^k u \text{ with } k\leq s.\]
We now extract $\pd_x^{s+1} u$ from $\ctr[\gbar]^{2q-1} \big(T^{q-1} \eta \big)$ using Equation (1.1) in \cite{Alb20}, which says for two (1,1)-double forms $\omega,\eta$, 
\begin{align*}
    \ctr[\gbar]^2(\omega \eta) = 2 \ctr[\gbar](\omega) \ctr[\gbar](\eta) - 2 \gbar^{ab}\, \gbar^{ij} \, \omega_{a,j}\, \eta_{i,b}.
\end{align*}
This together with Proposition \ref{prop: full contraction l form times 1 form} implies 
\begin{align*}
    &\ctr[\gbar]^{2q-1} (T^{q-1} \cdot \eta) \\
    =& -(2q-1)(2q-3) \ctr[\gbar]^{2q-2}\big(T^{q-1}\big) \ctr[\gbar](\eta) + \frac{1}{2}(2q-1) (2q-2) \ctr[\gbar]^2 \bigg[\Big( \ctr[\gbar]^{2q-3}\big(T^{q-1}\big) \cdot \eta \Big)\bigg]\\
    =& (2q-1) \ctr[\gbar]^{2q-2}\big(T^{q-1}\big) \ctr[\gbar](\eta) -(2q-1) (2q-2)  \bigg[\ctr[\gbar]^{2q-3}\big(T^{q-1}\big)\bigg]_{0,0} \cdot \eta_{0,0} \\
    &- (2q-1) (2q-2) \gbar^{ab} \gbar^{ij} \bigg[\ctr[\gbar]^{2q-3}\big(T^{q-1}\big)\bigg]_{a,j} \cdot \eta_{i,b}.
\end{align*}

It is important to note that $\ctr[\gbar](\eta) = \ctr[\gbar]\Big(\pd_x^{s-1}\Hess u\Big)$ and $\eta_{0,0} = \Big[\pd_x^{s-1}\Hess u\Big]_{0,0}$ both contribute to the highest order derivative $\pd_x^{s+1} u$. Consequently, combining the above results with Equation \eqref{eqn: scalar-2q derivative 2} yields:
\begin{align}
\begin{split} \label{eqn: scalar-2q derivative 3}
    \pd_x^s \Big(\widetilde{F}_q(g) \Big) =&\; q \pd_x u (n-2q+2)(2q)(2q-1) \cdot \\
    &\bigg(\big( s - (n-2q+3) \big) \ctr[\gbar]^{2q-2}\big(T^{q-1}\big) - (2q-2)s \Big[\ctr[\gbar]^{2q-3}\big(T^{q-1}\big)\Big]_{0,0} \bigg) \pd_x^{s+1} u\\
    &+ \text{terms involving } \pd_x^k u \text{ with } k\leq s. 
\end{split}
\end{align}

To analyze the boundary behavior of Equation \eqref{eqn: scalar-2q derivative 3}, we expand the contraction according to $u$ as follows:
\[\ctr[\gbar]^{2q-3}\big(T^{q-1}\big) = \left(-\frac{\kappa}{2}\right)^{q-1} \ctr[\gbar]^{2q-3}\big(\gbar^{2q-2}\big) + O(u) = \left(-\frac{\kappa}{2}\right)^{q-1}  \frac{n! (2q-2)!}{(n-2q+3)!} \, \gbar \,+\, O(u).\]
This implies $\Big[\ctr[\gbar]^{2q-3}\big(T^{q-1}\big)\Big]_{0,0}\Big|_{u=0} = \left(-\dfrac{\kappa}{2}\right)^{q-1}  \dfrac{n! (2q-2)!}{(n-2q+3)!} $, and 
\[\ctr[\gbar]^{2q-2}\big(T^{q-1}\big)\Big|_{u=0} = \left(-\frac{\kappa}{2}\right)^{q-1}  \frac{(n+1)! (2q-2)!}{(n-2q+3)!}.\]
Thus, restricting to the boundary $u=0$ and using $\pd_x u|_{x=0} = 1$, we obtain 
\begin{align*}
    &\pd_x^s \Big(\widetilde{F}_q(g) \Big)\Big|_{x=0} \\
    = & q (n-2q+2)(2q)(2q-1) \left(-\frac{\kappa}{2}\right)^{q-1} \frac{n! (2q-2)!}{(n-2q+3)!} \\
    &\hspace{30ex} \cdot \bigg( (n+1)s - (n-2q+3)(n+1)  - (2q-2)s \bigg) \cdot \pd_x^{s+1} u\Big|_{x=0} \\
    &+\text{terms involving } \pd_x^k u\Big|_{x=0} \text{ with } k\leq s\\
    = & q \left(-\frac{\kappa}{2}\right)^{q-1} \frac{n! (2q)!}{(n-2q+1)!}  \Big( s - (n+1) \Big) \cdot \pd_x^{s+1} u\Big|_{x=0} + \text{terms involving } \pd_x^k u\Big|_{x=0} \text{ with } k\leq s,
\end{align*} 
with $0 \leq s \leq n$, which leads to the following recurrence relation:
\begin{align} \label{eqn: singular yamabe 2q recurrence}
    q \left( -\frac{\kappa}{2}\right)^{q-1} \frac{n! (2q)!}{(n-2q+1)!}  &\Big( s - (n+1) \Big) \cdot \pd_x^{s+1} u\Big|_{x=0} \\
    = &\text{terms involving } \pd_x^k u\Big|_{x=0} \text{ with } k\leq s, \quad \text{ with } 0 \leq s \leq n. \notag
\end{align} 
This recurrence relation formally determines $\pd_x^{s+1} u|_{x=0}$ in terms of lower order derivatives of $u$ for $1 \leq s+1 \leq n+1$. However, a potential obstruction arises when attempting to solve for $\pd_x^{n+2} u|_{x=0}$. As in the Einstein case discussed by Graham \cite[\S2]{Gra17}, and analogous to the argument observed in the proof of Theorem \ref{thm: first approximation}, this obstruction can be resolved by introducing a logarithmic term of the form $x^{n+2} \log x$ in the expansion. It follows that the coefficient $\cL^\qq$ of the logarithmic term is uniquely determined. Consequently, there exists a solution 
\begin{equation}\label{eqn: sing Yamabe phg}
    u = x + u_2 x^2 + \cdots + u_{n+1} \,x^{n+1} + \cL^\qq x^{n+2} \log x
\end{equation}
to the equation
\[\widetilde{F}_q(g) = O(x^{n+2}\log x).\]
We refer to the coefficient $\cL^\qq$ as the \bfemph{singular Yamabe-$\boldsymbol{\qq}$ obstruction}. Under a conformal change of the background metric $\widetilde{\gbar} = \Omega^2 \gbar$, the singular Yamabe-$\qq$ obstruction transforms according to the relation $\widetilde{\cL}^\qq = \Omega^{-n-1}\cL^\qq$.

We then consider the singular Yamabe problem for linear combinations of scalar-$\qq$ curvatures given by  $\widetilde{F}_\beta(g) = \sum_q \beta_g \; \scal^\qq_g$. For $\beta = (\beta_1, \cdots, \beta_q)$, $2q\leq n$, Equation \eqref{eqn: singular yamabe 2q recurrence} becomes 
\[ n\big(s-(n+1)\big) \widetilde{B}_{1,2}(\beta, \kappa) \cdot \pd_x^{s+1} u\Big|_{x=0} = \text{terms involving } \pd_x^k u\Big|_{x=0} \text{ with } k\leq s, \quad 0 \leq s \leq n.\]

When $\widetilde{B}_{1,2}(\beta, \kappa) \neq 0$, there exists a formal solution of the form
\[u = x + u_2 x^2 + \cdots + u_{n+1} \,x^{n+1} + \cL^{(2p)} x^{n+2} \log x\]
which formally solves
\[\widetilde{F}_\beta(g) = \sum_q \beta_g \big(\scal_g^\qq - (n+1)\lambda^\qq \big) = O(x^{n+2}\log x). \]
\section{Conformally compact Lovelock fillings} \label{sec: CCL filling}
In \cite{GHS21}, the authors construct an invariant $\cI$ that necessarily vanishes for a conformally compact Einstein fillable pair. In this section, we will demonstrate that the same condition holds for a conformally compact Lovelock fillable pair $(X,h)$ when the scalar curvature of the corresponding conformally compact metric is bounded below by a constant depending on the dimension.

\subsection{Dirac operator and \texorpdfstring{$\cI$}{I}-invariant}
Henceforth, let $X$ denote a smooth, compact spin manifold of dimension $4k$ with boundary $M = \pd X$. Let $S(X)$ be the spinor bundle associated with $X$, and $D = D(X,g)$ be the Dirac operator acting on $S(X)$ with Atiyah-Patodi-Singer boundary conditions. The spinor bundle $S(X)$ decomposes into even and odd spinors~$S(X) = S^+(X) \oplus S^-(X)$, which induces a corresponding splitting of the Dirac operator into
\[D = \begin{pmatrix}
    0 & D^-\\
    D^+ & 0
\end{pmatrix}.\]

We say that a Riemannian metric $g$ on $X$ is a \bfemph{totally geodesic extension} of a Riemannian metric $h$ on $M = \pd X$ if $(M,h)$ is totally geodesic in $(X,g)$; that is, the second fundamental form of $M$ with respect to $g$ vanishes identically. Throughout, we assume that~$g$ is a totally geodesic extension of $h$ and that the Yamabe invariant $Y(M,h)$ is positive. We define the space
\[\cY^+(M) = \{h \text{ Riemannian metric on } M: Y(M,h) > 0 \}.\]
\begin{defn}
    The \bfemph{$\cI$-invariant} is defined as the index of the Dirac operator associated with $g$:
    \[\cI(X,h) = \ind(D^+(X,g)).\]
\end{defn}
According to Proposition 2.2 in \cite{GHS21}, the index $\cI$ is independent of the choice of the totally geodesic extension $g$, and depends only on $X$ and the path-connected components of~$h$ in $\cY^+(M)$. Thus, $\cI$ is a well-defined conformal invariant.

We will use the Gluing Lemma (\cite[Lemma 4.1]{GHS21}) and the Vanishing Lemma (\cite[Lemma 2.4]{GHS21}) to construct examples of non-fillable pairs. It is important to emphasize that these results are topological and do not depend on the Einstein structure; therefore, they also apply to Lovelock metrics. Before restating these lemmata, we introduce two types of connected sums of manifolds.

\begin{defn}
   The \bfemph{connected sum} of two closed manifolds $(M_1, h_1), (M_2, h_2)$ of dimension $n$, denoted $M_1 \# M_2$, is constructed by removing an open $n$-disk from each manifold and gluing the resulting boundaries $S^{n-1}$. The resulting metric is denoted by $h_1 \# h_2$. 
\end{defn}
It is known that positive scalar curvature is preserved under connected sums \cite[p.4192]{GHS21} when $n \geq 3$.
\begin{defn}
    Let $X_1, X_2$ be two manifolds with boundary of dimension $n+1$. We denote~$M_i = \pd X_i$, for $i=1,2$. Let $f_i: D^n \to M_i$ be embeddings of the closed disk $D^n$ into~$M_i$. The \bfemph{boundary connected sum} quotient $X_1 \#_\pd X_2 = X_1 \sqcup X_2 \sqcup (D^1 \times D^n)/\sim$, where $\sim$ is defined by the following identifications:
    \begin{align*}
        D^1 \times D^n \ni (-1,x) \sim f_1(x) \in M_1, \quad D^1 \times D^n \ni (1,x) \sim f_2(x) \in M_2.
    \end{align*}
\end{defn}
For visual depictions and a complete discussion of the boundary connected sum, we direct the reader to \cite[pp.4197--4199]{GHS21}.

There are two gluing lemmata corresponding to the two types of connected sums, which are derived from the index formula and involve tracking the sign of the eta invariant corresponding to the boundary’s orientation. 
\begin{lem}[Gluing Lemma] 
    Let $(X,g)$ be a smooth, compact Riemannian spin manifold of dimension $4k$ with boundary $\partial X = M_- \sqcup M_+$, where $X$ admits a decomposition 
    \[X = X_- \cup_M X_+\]
    and $M$ is a hypersurface of $X$. Furthermore, let $h_-$, $h$, and $h_+$ represent positive scalar curvature metrics on $M_-$, $M$, and $M_+$, respectively. Then 
    \[\cI(X, h_- \sqcup h_+) = \cI(X_-, h_- \sqcup h) + \cI(X_+, h \sqcup h_+).\]
    In particular, if $\partial X = \emptyset$, then $\Ahat(X) = \cI(X_-, h) + \cI(X_+, h).$
\end{lem}

\begin{lem}[Gluing Lemma for boundary connected sum]
    Let $X_1$, $X_2$ be compact spin manifolds of dimension $4k$ with $k>1$, each with non-empty boundary $M_i = \pd X_i$. Let $h_i$ be a positive scalar curvature metric on $M_i$. Then
    \begin{equation}
        \cI(X_1 \#_\pd X_2, h_1 \#_\pd h_2) = \cI(X_1, h_1) + \cI(X_2, h_2).
    \end{equation}
\end{lem}
If the totally geodesic extension $g$ has a positive scalar curvature, the Lichnerowicz formula indicates that the Dirac operator $D$ has no harmonic spinors, so the index $\cI$ vanishes. 
\begin{lem}[Vanishing Lemma] 
    Let $X$ be a smooth, compact spin manifold of dimension $4k$ with boundary $M$, where $M$ is equipped with a metric $h$ that has a positive Yamabe invariant. If the conformal class $[h]$ contains a representative $h_0$ that admits a totally geodesic extension~$g$ with positive scalar curvature, then $\cI(X,h) = 0$. 
\end{lem}

\subsection{Obstruction to CCL-filling}
The following two lemmata restate Lemma 2.1 and Theorem 2.2 from \cite{CQY04} in the context of conformally compact Lovelock metrics. The eigenfunction constructed in Lemma \ref{lem: efunction of Laplacian} will be used as a conformal factor for the conformal compactification of our Lovelock filling (see Theorem \ref{main thm: CCL filling obstruction}). The proofs of these results mirror those in the Einstein setting, with the exception that we impose a lower bound on the scalar curvature, a condition that is automatic when the metric is Einstein.

\begin{lem} \label{lem: efunction of Laplacian}
    Suppose that $(X^{n+1}, g)$ is a conformally compact Lovelock manifold with conformal infinity $[h]$. Then there exists a positive function $u$ that satisfies  
    \begin{enumerate}
        \item \begin{align} \label{eqn: CCL Laplacian eqn}
        \big(\Delta - (n+1)\big) u = 0 \text{ in } X, 
    \end{align} \label{eqn: CCL-filling expansion of u}
        \item $u$ has the following expansion near the boundary
        \begin{align}
        \begin{cases}
        \displaystyle \frac{1}{x} + \frac{\scal_h}{4n(n-1)}x + w^{(4)} x^3 + \text{(odd powers)} + w^{(n-1)} x^{n-2} + O(x^n) & n \text{ odd}\\[2ex]
        \displaystyle \frac{1}{x} + \frac{\scal_h}{4n(n-1)}x + w^{(4)} x^3 + \text{(odd powers)} + w^{(n)} x^{n-1} + O(x^n)  & n \text{ even},
    \end{cases}
    \end{align}
    where $w^{(i)}$ are all local invariants of Riemannian geometry of $(M,h)$ of order $i$.
    \end{enumerate}
\end{lem}
\begin{proof}
    The proof of this lemma in \cite[Lemma 2.1]{CQY04} for conformally compact Einstein metrics only uses the Fefferman-Graham expansion and thus holds true for a conformally compact Lovelock metric.
\end{proof}

The next lemma requires a lower bound on the scalar curvature. 
\begin{lem}\label{lem: CCL bound for scal}
    Suppose $(X^{n+1}, g)$ is a conformally compact Lovelock manifold satisfying that~$\scal_g \geq -n(n+1)$, and let $u$ be the eigenfunction obtained in Lemma \ref{lem: efunction of Laplacian} for a Yamabe metric $h$ of the conformal infinity $(M, [h])$. Then the manifold $(X^{n+1}, u^{-2}g)$ is compact with totally geodesic boundary $M$, and its scalar curvature satisfies $\scal_{u^{-2}g} \geq \frac{n+1}{n-1}  \, \scal_h$.
\end{lem}

\begin{proof}
    The conformal change formula of scalar curvature is given by
    \begin{align*}
        \scal_{u^{-2}g} 
        &=\scal_g u^2 + n(n+1) \big( 2u^2 -|du|_g^2 \big).
    \end{align*}
    It is pointed out in \cite[Equation (2.7)]{CQY04} that the expansion \eqref{eqn: CCL-filling expansion of u} implies 
    \begin{align*}
        u^2 -|du|_g^2 = \frac{\scal_h}{n(n-1)} + O(x^2).
    \end{align*}
    Under our assumption on the scalar curvature, we then have
    \begin{align*}
        \scal_{u^{-2}g} \geq n(n+1) \big( u^2 -|du|_g^2 \big) = \frac{n+1}{n-1} \, \scal_h + O(x^2).
    \end{align*}
    Once the above inequality is established, the proof is completed using the maximum principle and the Bochner formula, where $-\Delta \big( u^2 - |du|_g^2 \big) = 2 |D du - ug|^2$, following the same reasoning as in \cite[Theorem 2.2]{CQY04}.
\end{proof}

\begin{rmk}
    One might consider rescaling the metric $g$ to ensure that $\scal_g$ always satisfies the lower bound. However, such a rescaling would conflict with the asymptotically hyperbolic assumption, which requires the sectional curvatures to converge to $-1$. 
\end{rmk}

\begin{thm}[Obstruction to CCL Fillings] \label{main thm: CCL filling obstruction}
    Let $X$ be a smooth, compact Riemannian spin manifold with boundary, satisfying $\dim X = n+1 = 4k$ with $k \geq 2$. Suppose that $h$ is a Riemannian metric on the boundary $M = \pd X$. If the following conditions hold:
    \begin{enumerate} 
        \item the boundary metric $h$ admits a positive Yamabe invariant; 
        \item the index invariant $\cI(X,h) \neq 0$; 
    \end{enumerate} 
    then the pair $(X, h)$ does not admit any conformally compact Lovelock filling $(X, g)$ that satisfies $\scal_g \geq -n(n+1)$.
\end{thm}

\begin{proof}
     Suppose $X$ is fillable then with a conformally compact Lovelock metric $g$ satisfying that~ $\scal_g \geq -n(n+1)$. Then Lemma \ref{lem: CCL bound for scal} provides a conformal compactification $\widetilde{g}$ satisfying that 
     \begin{enumerate}
         \item the scalar curvature satisfies $\scal_{\widetilde{g}} \geq \frac{n+1}{n-1} \; \scal_h > 0$;
         \item the boundary $M$ is totally geodesic.
     \end{enumerate}
     Since $X$ is of even dimension, $\widetilde{g}$ has a smooth Fefferman-Graham expansion within a collar neighborhood of $M$. Then the Vanishing Lemma asserts that $\cI(X,h) = 0$ which contradicts Assumption (ii).
\end{proof}

\subsection{Non CCL-fillable pair}
As it is discussed in \cite[\S3--4]{GHS21}, the Gluing Lemma provides two types of examples of non-conformally compact Lovelock fillable pairs, analogous to the conformally compact Einstein filling problem. The first type involves perturbing a given pair  $(X,h)$ by attaching a closed spin manifold $Y$ with a non-vanishing $\Ahat$-genus, resulting in a non-fillable pair $(X\#Y,h)$. The second type involves perturbing the conformal infinity~$[h]$ by a metric $h'$ that belongs to $\cY^+(M)$, resulting in a non-fillable pair $(X,h \sqcup h')$.

\subsubsection{Type 1} (\cite[Proposition 3.3]{GHS21})
A smooth, compact Riemannian spin manifold~$X^{4k}$ with boundary can be viewed as the connected sum $(X - \mathring{D^{4k}}) \# \mathring{D^{4k}}$ where $D^{4k}$ is the $4k$-dimensional hyperbolic disk. The Gluing Lemma implies 
\[\cI(X,h) = \cI(X - \mathring{D^{4k}}, h \sqcup h_0) + \cI(D^{4k}, h_0) = \cI(X - \mathring{D^{4k}}, h \sqcup h_0),\]
where $h_0$ is the round metric on the boundary $S^{4k-1} = \pd D^{4k}$. Consider a closed spin manifold $Y$. 
The $\cI$-invariant of the connected sum of $X$ and $Y$ is given by 
\begin{align*}
    \cI(X \# Y, h) &= \cI(X - \mathring{D^{4k}}, h \sqcup h_0)+ \cI(Y - \mathring{D^{4k}}, h_0) = \cI(X,h) + \Ahat(Y).
\end{align*}
This suggests that given a conformal infinity $[h]$, we can always find a closed spin manifold~$Y$ with a nonvanishing $\Ahat$-genus so that $\cI(X \# Y, h) \neq 0$. Then the pair $(X \# Y, h)$ is not CCL-fillable by any metric $g$ that also satisfies $\scal_g \geq -n(n+1)$. 

\subsubsection{Type 2} (\cite[Proposition 4]{GHS21})
The Gluing Lemma for boundary connected sums enables us to construct a conformal infinity $h'$ of $X$ such that $\cI(X,h')\neq 0$. Suppose we start with a conformal infinity $h$ satisfying $\cI(X,h) = 0$. By taking the boundary connected sum of $X$ with $(D^{4k}, h_1)$, where $h_1$ is a metric differing from the round metric and satisfying that $\cI(D^{4k},h_1) \neq 0$, we obtain a new metric $h' = h \#_\pd h_1$. It follows that
\[\cI(X,h') = \cI(X \#_\pd D^{4k}, h \#_\pd h_1) = \cI(X,h) + \cI(D^{4k},h_1) \neq 0.\]
Thus, it is impossible for the pair $(X,h')$ to be CCL-fillable by a metric $g$ that also satisfies that  $\scal_g \geq -n(n+1)$.

%------------------------------
% appendix and references
%------------------------------
\bibliographystyle{alpha}
\bibliography{references}

\end{document}